\documentclass[11pt]{amsart}
\usepackage{inputenc,euscript,amssymb,geometry}
\usepackage{graphicx}
\usepackage{latexsym}
\usepackage{mathtools}
\usepackage{amsbsy,amsfonts,amscd}
\usepackage{color}

\newtheorem{theo}{Theorem} 

\newcommand{\D}{\displaystyle}
\newtheorem{lemma}{Lemma}[section]
\newtheorem{theorem}[lemma]{Theorem}
\newtheorem{prop}[lemma]{Proposition}
\newtheorem{corollary}[lemma]{Corollary}

\theoremstyle{definition}
\newtheorem{defi}[lemma]{Definition}
\newtheorem{nota}[lemma]{Notation}
\newtheorem{exttheo}{Theorem}

\newtheorem{rema}[lemma]{Remark}

\newcommand{\GG}{\mathbb{G}}
\newcommand{\HH}{\mathbb{H}}
\newcommand{\NN}{\mathbb{N}}
\newcommand{\RR}{\mathbb{R}}

\newcommand{\eps}{\varepsilon}

\newcommand{\HHH}{\mathcal{H}}

\newcommand{\KKK}{\mathcal{K}}

\newcommand{\QQQ}{\mathcal{Q}}

\newcommand{\tu}{\tilde{u}}

\begin{document}
\title[Focusing NLS on hyperbolic space]{Global existence, scattering and blow-up for the focusing NLS on the hyperbolic space}
\author[V. Banica]{Valeria Banica$^1$}
\address[V.Banica]{Laboratoire de Math\'ematiques et de Mod\'elisation d'\'Evry (UMR 8071)\\ D\'epartement de Math\'ematiques\\
Universit\'e d'\'Evry, 23 Bd. de France, 91037 Evry\\ 
France}
\email{Valeria.Banica@univ-evry.fr} 

\author[T. Duyckaerts]{Thomas Duyckaerts$^2$}
\address[T. Duyckaerts]{Laboratoire Analyse, G\'eom\'etrie et Applications (UMR 7539)\\D\'epartement de Math\'ematiques\\
Institut Galil\'ee, Universit\'e Paris 13, 99 Av. J.-B. Cl\'ement 
93430 Villetaneuse, France}
\email{duyckaer@math.univ-paris13.fr } 

\thanks{$^1$ Laboratoire de Math\'ematiques et de Mod\'elisation d'\'Evry (UMR 8071), Universit\'e d'Evry}\thanks{$^2$ LAGA (UMR 7539), Institut Galil\'ee, Universit\'e Paris 13, Sorbonne Paris Cit\'e}

\begin{abstract}
We prove global well-posedness, scattering and blow-up results for energy-subcritical focusing nonlinear Schr\"odinger equations on the hyperbolic space. We show in particular the existence of a critical element for scattering for all energy-subcritical power nonlinearities.  For mass-supercritical nonlinearity, we show a scattering vs blow-up dichotomy for radial solutions of the equation in low dimension, below natural mass and energy thresholds given by the ground states of the equation. The proofs are based on trapping by mass and energy, compactness and rigidity, and are similar to the ones on the Euclidean space, with a new argument, based on generalized Pohozaev identities, to obtain appropriate monotonicity formulas.
\end{abstract}

\date\today
\maketitle

\tableofcontents

\section{Introduction}

The nonlinear Schr\"odinger equations on manifolds have been intensively studied in the last decades. Most works concern local existence, blow-up in finite time, small data scattering and existence of wave operators. Recently, results on scattering for all solutions in the defocusing case were obtained on hyperbolic space \cite{BaCaSt08,IoSt09,IoPaSt12}, more general rotationally symmetric manifolds \cite{BaCaDu09}, flat manifolds such as exterior domains \cite{PlVe09,IvPl10,PlVe12,KiViZh12} and product spaces $\mathbb R\times \mathbb T^2, \mathbb R^n\times\mathbb T$ \cite{HaPa14,TzVi14}. Let us also note the recent work on long range effects on $\mathbb R\times \mathbb T^n$ \cite{HaPaTzVi13P}. The purpose of this work is to initiate the study of ``large" data - that is out of the perturbative framework of the small data theory - for focusing NLS on manifolds. More precisely, we are interested with the focusing NLS on the hyperbolic space $\HH^n$:
\begin{equation}
\label{NLS}
 i\partial_t u +\Delta_{\HH^n} u+|u|^{p-1}u=0,\quad u_{\restriction t=0}=u_0\in H^1(\HH^n),
\end{equation} 
where $n\geq 2$, $\Delta_{\HH^n}$ is the (negative) Laplace-Beltrami operator on $\HH^n$ and the power $p$ is energy-subcritical: $1<p< 1+\frac{4}{n-2}$ ($1<p<\infty$ if $n=2$).

It follows from Strichartz estimates that equation \eqref{NLS} is locally well-posed in the Sobolev space $H^1=H^1(\HH^n)$ \cite{Ba07}: for $u_0\in H^1$, there exists a unique maximal solution $$u \in C^0\left((-T_-(u_0),T_+(u_0)),H^1\right)$$ satisfying the following blow-up criterion:
\begin{equation}
\label{blupcrit}T_+(u_0)<\infty \Longrightarrow \lim_{t\to T_+(u_0)} \|u(t)\|_{H^1}=+\infty.\end{equation} 
The mass of a solution
\begin{equation}
\label{def_M}
M(u(t))=\int_{\HH^n} |u(t,x)|^2d\mu(x), 
\end{equation} 
(where $\mu$ is the standard measure on $\HH^n$) and its energy 
\begin{equation}
 \label{def_E}
E(u(t))=\frac{1}{2}\int_{\HH^n} |\nabla_{\HH^n} u(t,x)|^2d\mu(x)-\frac{1}{p+1}\int_{\HH^n}|u(t,x)|^{p+1}d\mu(x),
\end{equation} 
are conserved. 

In the defocusing case (equation \eqref{NLS} with a minus sign in front of the nonlinearity), it was proved in \cite{BaCaSt08,BaCaDu09,IoSt09} that all solutions with initial data in $H^1$ scatter to a solution of the linear Schr\"odinger equation in both time directions (see also \cite{IoPaSt12} for the energy-critical case in space dimension $3$). Note that this holds for all $p$ such that $1<p\leq 1+\frac{4}{n-2}$, in contrast with the Euclidean setting scattering results where a lower bound larger than $1$ is imposed on $p$. This is a consequence of the stronger long-time dispersion for the linear Schr\"odinger equation in $\HH^n$ compared to $\mathbb R^n$, which translates into a wider range of exponents for the Strichartz estimates. More precisely, global Strichartz estimates on $\HH^n$ are available for all exponents of Strichartz estimates on $\mathbb R^d,d\geq n$ \cite{BaCaSt08,IoSt09,AnPi09}. Using this fact the scattering results are proved for all the range of exponents $p$ allowed on $\mathbb 
R^d,d\geq n$, so for  $1<p\leq 1+\frac{4}{n-2}$.

In the focusing case, for the same reasons as above, scattering remains valid for small data in $H^1$, and wave operators also exist for all energy-subcritical $p$. However, solutions with larger initial data do not always scatter. If $p\geq 1+\frac{4}{n}$, blow-up in finite time may occur \cite{Ba07,MaZh07}. Furthermore, for any $p>1$, there exist nonzero time-periodic solutions of \eqref{NLS}. The aim of this article is to obtain sharp global existence, scattering and blow-up results in terms of geometric objects that are specific to $\HH^n$. Before stating our main results, we recall known ones on the focusing Schr\"odinger equation on $\RR^n$, $n\geq 1$:
\begin{equation*}
 i\partial_t u+\Delta_{\RR^n}u+|u|^{p-1}u=0,\quad u_{\restriction t=0}=u_0\in H^1(\RR^n),
\end{equation*} 
where $\frac{4}{n}+1< p$, and, if $n\geq 3$, $p<\frac{4}{n-2}+1$. Fixing $\mu<0$, the equation 
$$-\Delta f-\mu f=|f|^{p-1}f,\quad x\in \RR^n$$
has a unique radial, positive solution in $H^1(\RR^n)$ that we will denote by $R_{\mu}$. Let $s_c=\frac{n}{2}-\frac{2}{p-1}\in (0,1)$ be the critical Sobolev exponent, $M(u)$ and $E(u)$ be the invariant masses and energy, defined as in \eqref{def_M}, \eqref{def_E} with the integrals on $\RR^n$. Then (see \cite{Stu91,HoRo07,HoRo08,DuHoRo08,FaXiCa11,Gu13,AkNa13}):
\begin{exttheo}
\label{T:Rn}
Assume $1+\frac{4}{n}<p$, and $p<1+\frac{4}{n-2}$ if $n>3$.
Let $u_0\in H^1(\RR^n)$ be such that 
\begin{equation}
\label{sub_energy_Rn}
 E(u_0)^{s_c}M(u_0)^{1-s_c}<E(R_{\mu})^{s_c}M(R_{\mu})^{1-s_c}.
\end{equation} 
Let, for $t$ in the maximal interval of existence of $u$,
\begin{equation*}
 \delta(u(t))=\|\nabla u(t)\|_{L^2}^{s_c}\|u(t)\|_{L^2}^{1-s_c}-\|\nabla R_{\mu}\|_{L^2}^{s_c}\|R_{\mu}\|_{L^2}^{1-s_c}.
\end{equation*} 
Then $\delta(u(t))\neq 0$, and the sign of $\delta(u(t))$ is independent of $t$. Furthermore, 
\begin{enumerate}
 \item \label{I:scatt_Rn}If $\delta(u_0)<0$ then $u$ scatters in both time directions.
\item \label{I:blow-up_Rn}If $\delta(u_0)>0$ and either $u_0$ is radial and $p\leq 5$, or $\int |x|^2|u_0|^2<\infty$ then $u$ blows up in finite time.
\end{enumerate}
\end{exttheo}
Our aim is to obtain a scattering/blow-up dichotomy for equation \eqref{NLS}, in the spirit of the above theorem, under an optimal mass/energy threshold. Our main motivation is to clarify the influence of the geometry on the dynamics of focusing Schr\"odinger equations. 

Note that the choice of the ground state, i.e. of the parameter $\mu<0$ in the above  theorem is not relevant. Indeed,
$$R_{-1}(x)=|\mu|^{\frac{1}{p-1}}R_{\mu}(\sqrt{|\mu|} x),$$ 
and conditions \eqref{sub_energy_Rn}, as well as $\delta(u(t))$ do not depend on $\mu$. This scaling invariance is lost for the equation and $\HH^n$, and we will get a family of conditions, depending on a real parameter $\lambda$, which are not equivalent.

Let us recall that the spectrum of $\Delta_{\HH^n}$ is $\left(-\infty,\frac{(n-1)^2}{4}\right]$, and that the following sharp Poincar\'e-Sobolev inequality in $\HH^n$ is valid for all $f\in H^1(\HH^n)$ (see e.g. \cite{MaSa08}),
\begin{equation}
\label{PoiSob}
\left(\int_{\HH^n}|f|^{p+1}\right)^{\frac{2}{p+1}}\leq D(p,n)\left(\int_{\HH^n} |\nabla_{\HH^n} f|^2-\frac{(n-1)^2}{4}\int_{\HH^n} |f|^{2}\right).
\end{equation}
As a consequence, for all $\lambda<\frac{(n-1)^2}{4}$ and $f\in \HH^n$, the following inequality holds
\begin{equation}
\label{PoiSob_lambda}
\left(\int_{\HH^n}|f|^{p+1}\right)^{\frac{2}{p+1}}\leq D_{\lambda}\left(\int_{\HH^n} |\nabla_{\HH^n} f|^2-\lambda\int_{\HH^n} |f|^{2}\right).
\end{equation}
The best constant in \eqref{PoiSob_lambda} is attained for a positive, radial function $Q_{\lambda}\in H^1(\HH^n)$, solution of the equation
$$-\Delta Q_{\lambda}-\lambda Q_{\lambda}=|Q_{\lambda}|^{p-1}Q_{\lambda}$$
which we will call ground state (see \cite{MaSa08} and \S \ref{SS:Ground_states} for details). This ground state is not always known to be unique if $n=2$: when it is not we will denote by $Q_{\lambda}$ one of the ground states corresponding to $p$ and $\lambda$. In all cases, we let $\QQQ_{\lambda}$ the set of all ground states, that is the set of all solutions of the above equation that are also positive,  radial minimizers for \eqref{PoiSob_lambda}. 

Before stating our main results, we introduce some notations. If $f\in H^1(\HH^n)$, we denote by 
\begin{equation*}
 \|f\|_{\HHH_\lambda}^2=\int_{\HH^n} |\nabla_{\HH^n} f|^2d\mu-\lambda\int_{\HH^n} |f|^{2}d\mu,
\end{equation*} 
which, for any $\lambda<\frac{(n-1)^2}{4}$, is a norm on $H^1(\HH^n)$ equivalent to the usual $H^1$ norm. This is due to the spectrum of $\Delta_{\HH^n}$, implying $\|\nabla f\|_2^2\geq \frac{(n-1)^2}{4}\|f\|_2^2$. 
We define
$$E_\lambda(f)=\frac{1}{2}\|f\|^2_{{\HHH}_\lambda}-\frac{1}{p+1}\|f\|^{p+1}_{L^{p+1}}.$$
In particular $E_0=E$. Note that $E_{\lambda}(u(t))$ is independent of $t$ for any solution $u$ of \eqref{NLS}. We denote also
$$\delta_\lambda(f)=\|f\|_{{\HHH}_\lambda}^2-\|Q_{\lambda}\|^2_{{\HHH}_\lambda}.$$
When the ground state $Q_{\lambda}$ is not unique, $\delta_{\lambda}$ does not depend on the choice of $Q_{\lambda}\in \QQQ_{\lambda}$ (see \eqref{formule_Qp}, \eqref{formule_Dp}). 
\begin{theo}[Trapping and global existence]\label{T:global}
Let $n\geq 2$, $\lambda<\frac{(n-1)^2}4$ and $u_0\in H^1(\mathbb H^n)$. If $1< p<1+\frac 4{n-2}$ and $ E_\lambda(u_0)\leq E_\lambda(Q_{\lambda})$ then $\delta_\lambda(u(t))$ does not change sign. Moreover, under these hypothesis,
\begin{enumerate}
 \item \label{I:zero}If $\delta_\lambda(u_0)= 0$, then there exists $\theta\in \RR$ and an hyperbolic isometry $h$ such that $u_0=e^{i\theta}Q(h \cdot)$, $Q\in \QQQ_{\lambda}$.
\item \label{I:negative}If $\delta_\lambda(u_0)< 0$ then the solution $u$ is global in time.
\item \label{I:positive}If $\delta_\lambda(u_0)>0$ then the solution $u$ does not scatter in any time direction.
\end{enumerate}
\end{theo}
We refer to \S \ref{SS:prelim} below for the definition of the group of hyperbolic isometries.
\begin{rema}
The statement about global existence in Theorem \ref{T:global} is relevant only if $p\geq 1+\frac 4n$. If $1<p<1+\frac{4}{n}$, it follows from the Gagliardo-Nirenberg inequality on $\HH^n$ (which is the same than on $\RR^n$) that all solutions of \eqref{NLS} are global in time. 
\end{rema}
\begin{rema}
In the mass-critical case $p=1+\frac 4n$ a sharp global existence result based on the mass is known. On the one hand, global existence occurs for initial data of mass less than $1/C_{G-N}$ where $C_{G-N}$ is the best constant of Gagliardo-Nirenberg inequality on $\mathbb H^n$. On the other hand blow-up solutions can be constructed as in \cite{BaCaDu11,RaSz11} from the Euclidean ground state, hence of mass the inverse of the  best constant of Gagliardo-Nirenberg inequality on $\mathbb R^n$. Very recently the two constants have been proved to coincide, yielding a mass threshold for blow-up \cite{Mu14}.
\end{rema}
\begin{rema}
We note that in the case $p>1+\frac{4}{n}$, the analog of Theorem \ref{T:global} is valid on $\RR^n$ with almost the same proof. One can prove, for example, a global existence condition similar to Theorem \ref{T:global} \eqref{I:negative} for the NLS equation on $\RR^n$ depending on a parameter $\mu<0$ and using the ground states $R_{\mu}$ defined above. However, fixing an initial data $u_0$ and using the scaling to obtain the optimal value for $\mu$, one would exactly obtain the scale-invariant criteria of Theorem \ref{T:Rn}. In this sense, the conditions of Theorem \ref{T:global} are the natural analogs, for the hyperbolic space, of the conditions of Theorem \ref{T:Rn}.
\end{rema}

We conjecture that in the context of Theorem \ref{T:global}, with the stronger assumptions $E_\lambda(u_0)< E_\lambda(Q_{\lambda})$, and $p\geq 1+\frac{4}{n}$, solutions such that $\delta_{\lambda}(u_0)<0$ are global and scatter, whereas solutions such that $\delta_{\lambda}(u_0)>0$ blow up in finite time. This is false if $1<p<1+\frac{4}{n}$: all solutions are global, and there exist values of $\lambda$ and initial data $u_0$ such that $E_{\lambda}(u_0)<E_{\lambda}(Q_{\lambda})$, $\delta_{\lambda}(u_0)<0$, and the corresponding solution $u$ does not scatter (see Proposition \ref{P:stableGS} below). 

In this work we prove the conjecture in space dimensions $n=2$ and $n=3$, for radial data (i.e. depending only on the distance to the origin of $\HH^n$):
\begin{theo}
\label{T:main}
Assume $n\in\{2,3\}$, $p\geq 3$ if $n=2$,  $\frac{7}{3}\leq  p<5$ if $n=3$.  
 Let  $\lambda<\frac{(n-1)^2}4$ and $u_0\in H^1_{rad}(\mathbb H^n)$. Assume $E_\lambda(u_0)< E_\lambda(Q_{\lambda})$. Then 
\begin{enumerate}
\item \label{I:scattering}If $\delta_\lambda(u_0)< 0$ then the solution $u$ is global and scatters in both time directions.
\item \label{I:blow-up}If $\delta_\lambda(u_0)>0$, and 
$$\int_{\HH^n} r^2|u_0|^2\,d\mu<\infty\text{ or } 1+\frac{4}{n}<p \leq 5$$ 
then the solution $u$ blows up in finite positive and negative times.
\end{enumerate}
\end{theo}
Let us note that Theorem \ref{T:main} implies that for $p$, $n$ as in the theorem, the ground states are (orbitally) unstable: indeed, one can check as a consequence of this theorem that the solution with initial data $\alpha Q_{\lambda}$ blows up in finite time if $\alpha>1$, and scatters if $\alpha\in (0,1)$. 

Let us say a few words about the proof of Theorem \ref{T:main}. To prove the scattering result, we use the compactness-rigidity method initiated in \cite{KeMe06}. The \emph{compactness} step consists in proving the existence of a nonscattering solution of \eqref{NLS} with minimal energy. More precisely:
\begin{theo}[Existence of the critical element]\label{P:critical}Assume $1< p< 1+\frac 4{n-2}$ ($p>1$ if $n=2$) and $\lambda< \frac{(n-1)^2}4$.
There exists a global radial solution $v_c$ of equation \eqref{NLS}  such that
$$\{v_c(t,\cdot), t\in\mathbb R\}$$
has compact closure in $H^1(\mathbb H^n)$,
$$E_\lambda(v_c(0))\leq E_\lambda(Q_{\lambda}),\quad \|v_c(0)\|_{\HHH_\lambda}\leq \|Q_{\lambda}\|_{\HHH_\lambda},$$
and, for any $u_0\in H^1(\HH^n)$ radial, if 
$$E_\lambda(u_0)<  E_\lambda(v_c(0)),\quad \|u_0\|_{\HHH_\lambda}\leq \|Q_{\lambda}\|_{\HHH_\lambda},$$
then the solution $u$ of equation \eqref{NLS} scatters in both time directions. 
\end{theo}
We stated Theorem \ref{P:critical} in a radial setting. A nonradial version is available (see Proposition \ref{P:critical2} p. \pageref{P:critical2}). 
\begin{rema}
Note that $v_c$ exists in all dimensions and for all energy-subcritical exponent $p$. Again, this contrasts with the Euclidean case where for $p$ close to one there is no small data scattering, and thus no critical solution in the above sense.
\end{rema}
The proof of Theorem \ref{P:critical} follows the line of the corresponding proof on $\RR^n$ (see \cite{Ke06,KeMe06,TaViZh08,HoRo07}). The main ingredient of the proof is a profile decomposition adapted to the energy-subcritical equation \eqref{NLS}. We construct this profile decomposition in Section \ref{SS:profile}, using Fourier analysis on the hyperbolic spaces, in the spirit of the analogous construction in the energy-critical setting, given in \cite{IoPaSt12}. 

The \emph{rigidity} step in the proof of Theorem \ref{T:main} \eqref{I:scattering} consists in proving that the critical element $v_c$ given by Theorem \ref{P:critical} (under the assumptions on $n$ and $p$ in Theorem \ref{T:main}), satisfies $E_{\lambda}(v_c)=E_{\lambda}(Q_{\lambda})$ (see Proposition \ref{prop-rigidity} p. \pageref{prop-rigidity}). Similarly to the Euclidean case, we use a localized version of the following virial-type identity:
\begin{equation}
\label{virial_intro}
\partial_t^2\int |u(t,x)|^2r^2\,d\mu(x)=G(u(t)), 
\end{equation} 
where, $r$ is the distance to the origin of $\HH^n$, and, if $f$ is radial,
\begin{multline*}
G(f)=8\|f\|_{\mathcal H}^2+2(n-1)(n-3)\int_{\mathbb H^n}|f|^2\frac{r\cosh r-\sinh r}{\sinh^3r}\,d\mu(x)\\
-\frac{4(p-1)}{p+1}\int_{\mathbb H^n}|f|^{p+1}\left(1+(n-1)\,\frac{r\cosh r}{\sinh r}\right)\,d\mu(x).
\end{multline*}
A crucial property of $G$ is that it is positive for solutions satisfying the assumptions of Theorem \ref{T:main} \eqref{I:scattering}. However, unlike in the Euclidean setting where the analogous property follows quite easily from the characterization of the ground states $R_{\mu}$ as maximizers for the Gagliardo-Nirenberg inequality and the trapping of solutions below the ground state mass and energy,  the proof of this property is quite intricate (see Section \ref{S:rigidity}). The key new ingredient of this proof is a generalized Pohozaev identity satisfied by the minimizers of $G(f)$ under the constraint $E_{\lambda}(f)=E_{\lambda}(Q_{\lambda})$. It is in this part of the proof that the assumption $n=2,3$ is needed.

We think that the radiality assumption and the assumption $n=2,3$ are technical, and that Theorem \ref{T:main} remains valid for any $n\geq 2$, without symmetry, provided $p\geq 1+\frac{4}{n}$.  On the other hand, the hypothesis $p\geq 1+\frac{4}{n}$ is crucial, as emphasized by the following proposition:
\begin{prop}\label{P:stableGS}
 If $n\geq 3$ and $1<p<1+\frac{4}{n}$, there exists $\lambda<\frac{(n-1)^2}{4}$ such that $E_{\lambda}(v_c(0))<E_{\lambda}(Q_{\lambda})$. 
\end{prop}
We do not know what is $v_c$ in this case (a ground state $Q_{\nu}$ with $\nu\neq \lambda$, or some other type of solution). 

Proposition \ref{P:stableGS} follows from the fact that if $p<1+\frac{4}{n}$, for some $\lambda<\frac{(n-1)^2}{4}$  the solution $e^{it\lambda}Q_{\lambda}$ is stable (in the set of orbits of the minimizers of the energy at this mass). It is an open question if all ground states are stable if $p<1+\frac{4}{n}$ 
(this is stated  in \cite{ChMa10}, but with a gap in the proof, see Remark \ref{R:ChMa} below).

The proof of the blow-up part of Theorem \ref{T:main} follows the classical proof of the so-called Glassey criterion on $\RR^n$ \cite{VlPeTa71,Gl77}, using the virial identity \eqref{virial_intro}. For this, we prove that $G$ is negative for solutions of \eqref{NLS} satisfying the assumptions of Theorem \ref{T:main} \eqref{I:blow-up}, using arguments that are similar to the ones of the proof of the positivity of $G$ in the other regime of Theorem \ref{T:main}. 

We refer to \cite{Ba07} and \cite{MaZh07} for other blow-up criteria, recalled in Proposition \ref{P:B-up} p. \pageref{P:B-up}. Note that in these criteria, the threshold given by $Q_{\lambda}$ does not appear. In \cite{MaZh07} a variant of virial identity \eqref{virial_intro} is used, based on another weight than $r^2$. Unfortunately this different weight does not seem useful in the setting of Theorem \ref{T:main}. It is also possible to construct blow-up solutions with an explicit behavior, starting from the Euclidean ground state: see \cite{BaCaDu09}, \cite{RaSz11}, \cite{Bo12}, \cite{Go13} for the construction of conformal and log-log type blow-up solutions on $\HH^n$ or on related manifolds.

The outline of the article is as follows. In Section \ref{S:LGWP}, we prove Theorem \ref{T:global}, after some preliminaries and reminders on the hyperbolic space, Cauchy theory for equation \eqref{NLS} and ground states. Section \ref{S:critical} is dedicated to the existence of the critical solution (Theorem \ref{P:critical} and its nonradial analog). Following a standard scheme, we construct an adapted profile decomposition (see \S \ref{SS:profile}), which follows from an improved Sobolev inequality, proved in \S \ref{SS:spaces}. The critical solution  is constructed in \S \ref{SS:critical}. Proposition \ref{P:stableGS} is proved in \S \ref{SS:mass_subscritical}. In Section \ref{S:rigidity}, we conclude the proof of Theorem \ref{T:main} \eqref{I:scattering} (scattering) by proving the rigidity part of the argument. Section \ref{section:blow-up} concerns the proof Theorem \ref{T:main} \eqref{I:blow-up} (blow-up).

\subsection*{Acknowledgment}
The authors would like to thank Beno\^{i}t Pausader for very useful clarifications on spectral projectors on the hyperbolic space. Both authors were partially supported by the French ANR project SchEq ANR-12-JS-0005-01. T.D. was partially supported by ERC Grant no. 257293 Dispeq and ERC Advanced Grant no. 291214 BLOWDISOL

\section{Local and global well-posedness}
\label{S:LGWP}
\subsection{Notations and preliminaries on the hyperbolic space}
\label{SS:prelim}

Recall that $\HH^n$ is defined as 
$$ \HH^n=\Big\{ x\in \RR^{n+1}\;:\; [x,x]=1\text{ and }x_0>0\Big\},$$
where $[\cdot,\cdot]$ is the bilinear form
$$[x,y]=x_0y_0-x_1y_1-\ldots-x_ny_n$$
on $\RR^{n+1}$. The hyperbolic space $\HH^n$ is
endowed with the metric  $g$ induced by the Minkowski metric $-(dx_0)^2+(dx_1)^2+\ldots+(dx_n)^2$. We will denote by $0$ the origin $(1,0,\ldots,0)$ of $\HH^n$, and $d\mu$ the induced measure.

We shall use often radial coordinates on the hyperbolic space, $x=(\cosh r,\sinh r\,\omega)$ where $r=d(x,0_{\mathbb H^n})$, $\omega\in\mathbb S^{n-1}$. In such coordinates, the Laplacian writes
$$\Delta_{\HH^n} =\partial_r^2+(n-1)\frac{\cosh r}{\sinh r}\partial_r+\frac 1{\sinh ^2r}\Delta_{\mathbb S^{n-1}}.$$ 
To lighten notations, we will often write $\Delta$ instead of $\Delta_{\HH^n}$.

We denote by $\GG=SO(n,1)$ the group of hyperbolic isometries, that is the group of $(n+1)\times (n+1)$ matrices that leave the form $[\cdot,\cdot]$ invariant. For any $h\in \GG$, the mapping $x\mapsto h\cdot x$ restricts to an isometry of $\HH^n$. The group $\GG$ acts transitively on $\HH^n$.

We introduce the following notation, which is the quadratic form associated to the so-called shifted Laplacian on hyperbolic space, whose bottom of the spectrum is zero,
\begin{equation}
\label{def_norm}
 \|f\|_{\HHH}^2=\int_{\HH^n} |\nabla_{\HH^n} f|^2-\frac{(n-1)^2}{4}\int_{\HH^n} |f|^{2}.
\end{equation} 
By \eqref{PoiSob}, $\|\cdot\|_{\HHH}$ is a norm on $C_c^{\infty}(\HH^n)$. We will denote by $\HHH$ the closure of $C_c^{\infty}(\HH^n)$ in $L^{p+1}(\HH^n)$ for the norm $\|\cdot\|_{\HHH}$. It is a Hilbert space which is included in $L^{p+1}$. 

\subsection{Cauchy theory}
We give here some results related to well-posedness and scattering for equation \eqref{NLS}. We omit most of the proofs, that are classical.
\subsubsection{Strichartz estimates on the hyperbolic space}
\label{SSS:Strichartz}
We will denote by $q'$ the conjugate exponent of $q\in [1,\infty]$.
We recall from \cite{BaCaSt08,AnPi09,IoSt09} the wider range of Strichartz estimates on the hyperbolic space:
\begin{theorem}
 \label{T:Strichartz}
 Let, for $j=1,2$, 
 $$(q_j,r_j)\in \Big\{ (q,r)\in [2,\infty)\times (2,\infty)\;:\; \frac{2}{q}\geq \frac n2 -\frac nr\Big\}\cup\Big\{(\infty,2)\Big\}.$$
 If $u_0\in L^2(\HH^n)$, $F\in L^{q_2'}(\RR,L^{r_2'}(\HH^n))$, then, denoting by $u$ the solution of 
 $$i\partial_t u+\Delta u=F,\quad u_{\restriction t=0}=u_0,$$
 we have
 $$\|u\|_{L^{q_1}(\RR_t,L^{r_1}(\HH^n))}\leq C\|u_0\|_{L^2}+C\|F\|_{L^{q_2'}(\RR,L^{r_2'}(\HH^n)}.$$
\end{theorem}
Let $I$ be an interval. If $1<p\leq 1+\frac{4}{n}$, we define 
\begin{align*}
S^0(I)&=L^{p+1}(I,L^{p+1}(\HH^n))& S^1(I)&=\Big\{u\in S^0(I)\;:\; \nabla u\in S^0(I)\Big\}\\
N^0(I)&=L^{\frac{p+1}{p}}(I,L^{\frac{p+1}{p}}(\HH^n)) & N^1(I)&=\Big\{u\in N^0(I)\;:\; \nabla u\in N^0(I)\Big\}. 
\end{align*}
Note that $\frac{p+1}{p}=(p+1)'$.

If $1+\frac{4}{n}\leq p<1+\frac{4}{n-2}$, we let (following \cite{FaXiCa11})
$$ a=\frac{2(p-1)(p+1)}{4-(n-2)(p-1)},\quad b=\frac{2(p-1)(p+1)}{n(p-1)^2+(n-2)(p-1)-4},\quad q=\frac{4(p+1)}{n(p-1)},$$
(so that $pb'=a$), and 
\begin{align*}
S^0(I)&=L^{a}(I,L^{p+1}(\HH^n))& S^1(I)&=\Big\{u\in L^q(I,L^{p+1}(\HH^n))\;:\; \nabla u\in L^q(I,L^{p+1}(\HH^n))\Big\}\\
N^0(I)&=L^{b'}(I,L^{\frac{p+1}{p}}(\HH^n)) & N^1(I)&=\Big\{u\in L^{q'}(I,L^{\frac{p+1}{p}}(\HH^n))\;:\; \nabla u\in L^{q'}(I,L^{\frac{p+1}{p}}(\HH^n))\Big\}. 
\end{align*}
One easily checks that the definitions coincide when $p=\frac{4}{n}+1$. 

If $I=(a,b)$, we will write $S^0(a,b)$ instead of $S^0((a,b))$, and similarly for $S^1$, $N^0$, $N^1$. 

\begin{prop}
 \label{P:Stric}
If $t_0\in \RR\cup\{\pm\infty\}$,
 \begin{gather}
  \label{Stric1}
  \left\|e^{it\Delta}u_0\right\|_{S^j(\RR)\cap L^{\infty}(\RR,H^1)} \leq C\|u_0\|_{H^1},\quad j=0,1,\\
\label{Stric2}  
\left\|\int_{t_0}^t e^{i(t-s)\Delta}f(s)ds\right\|_{S^0(\RR)}\leq C\|f\|_{N^0(\RR)}\\
\label{Stric3}
\left\|\int_{t_0}^{t} e^{i(t-s)\Delta}f(s)ds\right\|_{S^1(\RR)\cap L^{\infty}(\RR,H^1)}\leq C\|f\|_{N^1(\RR)}\\
 \notag\text{ and }t\mapsto \int_{t_0}^{t} e^{i(t-s)\Delta}f(s)ds\in C^0(\RR,H^1), \text{ if }f\in N^1(\RR),\\
 \label{Stric3''}
 \left\|\int_0^{+\infty}e^{-is\Delta} f(s)ds\right\|_{H^1}\leq C\|f\|_{N^1(0,\infty)}.
 \end{gather}
\end{prop}
\begin{prop}
 \label{P:Holder}
\begin{align}
 \label{Holder1}
 \left\||u|^{p-1}u-|v|^{p-1}v\right\|_{N^0(I)}&\leq C\|u-v\|_{S^0(I)}\left(\|u\|^{p-1}_{S^0(I)}+\|v\|^{p-1}_{S^0(I)}\right)\\
\label{Holder2}
 \left\||u|^{p-1}u\right\|_{N^1(I)}&\leq C\|u\|_{S^1(I)}\|u\|^{p-1}_{S^0(I)}.
\end{align}
 \end{prop}
\begin{rema}
 If $p>2$, we can of course obtain a Lipschitz bound similar to \eqref{Holder1} for the $N^1$-norm.
\end{rema}
\begin{proof}[Sketch of proof of Proposition \ref{P:Stric}]
 The inequalities are obtained from Theorem \ref{T:Strichartz} as follows. If $p<1+\frac{4}{n}$, the pair $(p+1,p+1)$ is not an Euclidean admissible but it enters the wider range of Strichartz exponents on the hyperbolic space from Theorem \ref{T:Strichartz}. This yields inequalities \eqref{Stric1}-\eqref{Stric3''} if $p<1+\frac{4}{n}$. 
 
 In the case $1+\frac{4}{n}\leq p<1+\frac4{n-2}$,
 the Proposition follows from the Euclidean-type Strichartz estimates and Sobolev inequalities, as for instance in \cite{FaXiCa11}. The only delicate point is the Strichartz estimate for non-admissible couples \eqref{Stric2}; these can be obtained from dispersion in the spirit of Lemma 2.1 in \cite{CaWe92} (see also \cite{Fo05}).
 We sketch it (with $t_0=0$) for completeness. By Lemma 3.3 of \cite{IoSt09}, $$ \forall t\neq 0,\quad \left\|e^{it\Delta}u_0\right\|_{L^{p+1}}\leq \frac{C}{|t|^{\frac{2}{q}}}\left\|u_0\right\|_{L^{\frac{p+1}{p}}}.$$ Thus $$ \forall t\neq 0,\quad \left\|\int_0^{t}e^{i(t-s)\Delta}f(s)\,ds\right\|_{L^{\frac{p+1}{p}}}\leq C\int_0^{t} \frac{1}{|t-s|^{\frac{2}{q}}}\|f(s)\|_{L^{\frac{p+1}{p}}}\,ds,$$ and the result follows from the classical Riesz potential inequality (see e.g. Ch. 5 of \cite{Ste70Bo}).

 \end{proof}

Proposition \ref{P:Holder} follows immediately from H\"older inequality and we omit it. 

In view of Propositions \ref{P:Stric} and \ref{P:Holder}, the well-posedness of equation \eqref{NLS} in $H^1$ is classical (see \cite{Ka87}).
Recall that a solution $u$ of \eqref{NLS}, defined on a maximal interval of existence $(T_-(u),T_+(u))$ satisfies the following blow-up criterion
\begin{equation*}
T_+(u)<\infty\Longrightarrow \lim_{t\to T_+(u)} \|u(t)\|_{H^1}=0.
\end{equation*}

\subsubsection{Scattering results}
\begin{prop}[Existence of wave operators]
\label{P:wave_op}
 Let $v_0\in H^1(\HH^n)$. Then there exists a solution $u$ of \eqref{NLS} such that $T_+(u)=+\infty$ and 
$$\lim_{t\to\infty} \left\|e^{it\Delta}v_0-u(t)\right\|_{H^1}=+\infty.$$
\end{prop}
This follows by a fixed point in the closed subset of $S^0(T,+\infty)$:
$$B_{T,\eps}=\left\{u\in S^1(T,+\infty)\cap S^0(T,+\infty)\;:\;\|u\|_{S^1(T,+\infty)}+\|u\|_{S^0(T,+\infty)}\leq \eps\right\},$$
for $T$ large, $\eps>0$ small, using again Propositions \ref{P:Stric} and \ref{P:Holder}. We omit the details of the classical proof.

\begin{prop}[Sufficient condition for scattering]
\label{P:scatt}
 Let $u$ be a solution of \eqref{NLS} with maximal time of existence $T_+$ and such that 
$$ \|u\|_{S^0(0,T_+)}<\infty.$$
Then $T_+=+\infty$ and $u$ scatters forward in time to a linear solution: there exists $v_0\in H^1$ such that
\begin{equation*}
\lim_{t\to +\infty}\left\|e^{it\Delta}v_0-u(t)\right\|_{H^1}=0.
\end{equation*} 
\end{prop}
We skip the standard proof.

\begin{prop}[Long time perturbation theory]
\label{P:LTPT}
 Let $M>0$. There exists constants $\eps_0>0$, $C>0$ depending on $M$ with the following properties. Let $0<T\leq \infty$ and
$$u_0\in H^1,\quad \tu\in C^0((0,T),H^1)\cap S^0(0,T),\quad e\in N^0(0,T)$$
such that
$$ i\partial_t \tilde{u}+\Delta \tilde{u}+|\tu|^{p-1}\tu=e.$$
Assume 
$$\|\tu\|_{S^0(0,T)}\leq M,\quad \|e\|_{N^0(0,T)}+\left\|e^{it\Delta}(u_0-\tilde{u}(0))\right\|_{S^0(0,T)}=\eps\leq \eps_0.$$
Then the solution $u$ of \eqref{NLS} with initial data $u_0$ is defined on $(0,T)$ and $\|u-\tilde{u}\|_{S^0(0,T)}\leq C\eps$.
\end{prop}
 This type of result that goes back to \cite[Lemma 3.10]{CoKeStTaTa08}, is by now standard. In the case $p>1+\frac{4}{n}$, in view of the Strichartz estimates of Proposition \ref{P:Stric}, the proof is exactly the same as in \cite[Proposition 4.7]{FaXiCa11} (simply replacing $\RR^n$ by $\HH^n$). In the case $1<p\leq 1+\frac{4}{n}$, it can be easily adapted, using Propositions \ref{P:Stric} and \ref{P:Holder}. We apply the preceding proposition with $\tilde{u}=e^{it\Delta}u_0$ to get:

\begin{corollary}[Small data theory]
\label{C:small_data}
There exists $\eps_1$ such that if $u_0\in H^1$ satisfies 
$$\|e^{it\Delta}u_0\|_{S^0(\RR)}=\eps\leq \eps_1$$
then the corresponding solution $u$ of \eqref{NLS} is global and satisfies 
$$\|u-e^{it\Delta}u_0\|_{S^0(\RR)}\leq C\eps^p.$$
\end{corollary}

\subsection{Ground states on the hyperbolic space}
\label{SS:Ground_states}
We review here results on ground states for NLS on the hyperbolic space and additional variational properties. Most of these results come from \cite{MaSa08}, see also \cite{Wa14}. See \cite{MuTa98} for a previous work in space dimension $2$ and \cite{ChMa10} for similar existence results.

Consider the equation on $\HH^n$ 
\begin{equation}
\label{Eplambda}
\Delta_{\HH^n} f+\lambda f +|f|^{p-1}=0,
\end{equation}
Positive solutions to \eqref{Eplambda} can be constructed as solution to the following minimizing problems:
\begin{equation}
\label{minimizer}
\frac{1}{D_{\lambda}}=\min_{f\in H^1\setminus\{0\}}\frac{\|f\|^2_{{\HHH}_\lambda}}{\|f\|^2_{L^{p+1}}}
\end{equation}
Then (Theorems 5.1 and 5.2 of \cite{MaSa08}): 
\begin{theorem}
\label{T:minimizers}
The minimizing problem \eqref{minimizer} has a solution if
\begin{equation}
\label{hyp_p}
(n=2\text{ or }1<p< 1+\frac{4}{n-2}) \text{ and }\lambda < \frac{(n-1)^2}{4}.
\end{equation}
In this case, any minimizer is radial up to hyperbolic symmetries, positive up to multiplication by a unit complex number,  and satisfies equation \eqref{Eplambda}.  
\end{theorem}

We will denote by $\QQQ_{\lambda}$ the set of positive, radial minimizers for \eqref{minimizer} that are solutions to equation \eqref{Eplambda}, and $Q_{\lambda}$ an arbitrary element of $\QQQ_{\lambda}$. We note that $Q_{\lambda}$ is not always known to be unique (see below), however our statements will never depend on the choice of $Q_{\lambda}$.

Let $Q\in \QQQ_{\lambda}$. Multiplying \eqref{Eplambda} by $\varphi(\eps r)Q$ (where $\varphi$ is a radial smooth compactly supported function equal to $1$ around $0$), integrating by parts and letting $\eps\to 0$, we get 
\begin{equation}
 \label{formule_Qp}
\forall Q\in \QQQ_{\lambda},\quad
 \left\|Q\right\|_{{\HHH}_\lambda}^2=\left\|Q\right\|^{p+1}_{L^{p+1}}
\end{equation} 
As a consequence, 
\begin{equation}
\label{formule_Dp}
\forall Q\in \QQQ_{\lambda},\quad
D_{\lambda}=\frac{\left\|Q\right\|^2_{L^{p+1}}}{\left\|Q\right\|^2_{{\HHH}_\lambda}}=\left\|Q\right\|^{1-p}_{L^{p+1}}.
\end{equation} 
In particular, the values of $E_{\lambda}(Q)$, $\|Q\|_{\HHH_{\lambda}}$ and $\|Q\|_{L^{p+1}}$ do not depend on the choice of $Q$ in $\QQQ_{\lambda}$. 


The following theorem follows from Theorems 1.2 and 1.3 of \cite{MaSa08}:
\begin{theorem}[Uniqueness]
\label{T:uniqueness}
Assume \eqref{hyp_p}. If $n\geq 3$, or $n=2$ and $\lambda\leq \frac{2(p+1)}{(p+3)^2}$ equation \eqref{Eplambda}  has only one positive solution up to hyperbolic isometries. 
\end{theorem}
\begin{rema}
Uniqueness in the case $n=2$, $\frac{2(p+1)}{(p+3)^2}<\lambda\leq \frac{1}{4}$ is an open question.
\end{rema}
\begin{theorem}[Nonexistence]
\label{T:nonexistence}
If $p>1$, $\lambda \geq \frac{(n-1)^2}{4}$, then equation \eqref{Eplambda} has no positve solution in $H^1$. 
\end{theorem}

Let us mention that for the critical value $\lambda=\frac{(n-1)^2}{4}$, equation \eqref{Eplambda} has a solution which is in $\HHH$ (see \eqref{def_norm}), but not in $H^1$, and solution to the minimization problem \eqref{minimizer}. For $\lambda>\frac{(n-1)^2}{4}$ the nonexistence theorem \ref{T:nonexistence} remains valid if $H^1$ is replaced by $\HHH$.

We next give a result that is specific to the mass-critical case, and will be needed in the proof of Proposition \ref{P:stableGS}.
\begin{prop}
\label{P:mass-sub}
Assume $1<p<1+\frac{4}{n}$. Then there exists $\alpha_0>0$ such that for all $\alpha>\alpha_0$, the infimum
\begin{equation}
 \label{min2}
 \inf_{\substack{u\in H^1(\HH^n)\\ \|u\|^2_{L^2}=\alpha^2}} \frac 12\|u\|^2_{\HHH}-\frac{1}{p+1}\|u\|^{p+1}_{L^{p+1}}=e(\alpha)
\end{equation} 
is attained by a radial, positive function. If $n\geq 3$, this function is equal to $Q_{\lambda}$ for some $\lambda<\frac{(n-1)^2}{4}$. Finally, any radial minimizing sequence converges (up to a subsequence) to a radial minimizer.
\end{prop}
\begin{proof}
First we note that Gagliardo-Nirenberg inequality implies the infimum in \eqref{min2} to be finite if $1<p<1+\frac{4}{n}$. By a rearrangement procedure (\cite{Dr05} and \cite[Section 3]{ChMa10}) or moving planes technics \cite[Section 2]{MaSa08} , the infimum in \eqref{min2} can be restricted to $H^1$ radial functions such that $\|u\|_{L^2}=\alpha$. Using as in \cite{ChMa10} the change of functions $v=\left(\frac{\sinh r}{r}\right)^{\frac{n-1}{2}}u$, we are reduced to minimize 
 $$ \int_{\RR^n} \frac 12 |\nabla v|^2+\frac{(n-1)(n-3)}{8}\int_{\RR^n} |v|^2\frac{r^2-\sinh^2 r}{r^2\sinh^2 r}-\frac{1}{p+1}\int_{\RR^n} |v|^{p+1}\left(\frac{r}{\sinh r}\right)^{\frac{(p-1)(n-1)}{2}},$$
 on all radial function in $H^1(\RR^n)$ such that $\|v\|_{L^2}=\alpha$. Since $\|\nabla |u|\|_{L^2}\leq \|\nabla u\|_{L^2}$, with strict inequality if $u$ is not positive up to a constant factor, minimizers are positive (up to a constant factor). Using, as in \cite{ChMa10}, the concentration-compactness method, or simply the compactness of the radial embedding of $H^1$ in $L^p$, it is easy to prove that if $e(\alpha)<0$, the infimum is attained. Fixing $u\in H^1(\HH^n)$ with $\|u\|_{L^2}=1$, we obtain
 $$\lim_{\alpha\to\infty} \frac{\alpha^2}{2}\|u\|^2_{\HHH}-\frac{\alpha^{p+1}}{p+1}\|u\|^{p+1}_{L^{p+1}}=-\infty.$$
 Thus $e(\alpha)$ is negative for large $\alpha$, concluding the proof of the existence of a minimizer $S_{\alpha}$ for \eqref{min2}. The compactness of minimizing sequences also follows. Since $S_{\alpha}$ is solution to the minimization problem \eqref{min2}, there exists a Lagrange multiplier $\lambda$ such that 
 $$-\Delta S_{\alpha}-\lambda S_{\alpha}=|S_{\alpha}|^{p-1}S_{\alpha}.$$
By the considerations above, we can assume that $S_{\alpha}$ is radial and positive. If $n\geq 3$, using Theorems \ref{T:nonexistence} and \ref{T:uniqueness}, we obtain $\lambda<\frac{(n-1)^2}{4}$ and $S_{\alpha}=Q_{\lambda}$, concluding the proof.
 \end{proof}
\begin{rema}
 In \cite[Section 5]{ChMa10}, it is claimed that the infimum \eqref{min2} is attained for all $\alpha>0$. This cannot be true, since it would contradict the small data scattering for equation \eqref{NLS} for $1<p<1+\frac{4}{n}$, proved in \cite[Section 4]{BaCaSt08}. Note that for small $\alpha>0$, one can prove (using Poincar\'e-Sobolev \eqref{PoiSob} and  Gagliardo-Nirenberg inequalities) that $e(\alpha)=0$, whereas it is claimed and used in the proof of \cite{ChMa10} than $e(\alpha)<0$ for any $\alpha>0$.
\end{rema}
\begin{rema}
\label{R:ChMa}
One can deduce from Proposition \ref{P:mass-sub}, following \cite{CaLi82}, the orbital stability of the set of all solutions $e^{it\lambda}Q_{\lambda}$ of \eqref{NLS}, with $Q_{\lambda}$ minimizer for \eqref{min2}, that is of mass $\alpha=\|Q_{\lambda}\|_{L^2}$ (see \cite[Section 6]{ChMa10}). Note that the proof of Proposition \ref{P:mass-sub} does not imply that any $Q_{\lambda}$ is a minimizer for the problem \eqref{min2}. In particular, Proposition \ref{P:mass-sub} and the method of \cite{CaLi82} do not yield stability for all ground states solutions $e^{it\lambda}Q_{\lambda}$ as seems to say \cite[Proposition 6.3]{ChMa10}. We refer to \cite{LaSuSo14Pa} for the study of ground states stability for wave maps on the hyperbolic plane: in this case also the situation is quite different from the Euclidean setting.
Let us also mention that uniqueness of minimizers for \eqref{min2} and uniqueness of a minimal mass ground state are open questions. A similar issue appears in the context of combined power-type nonlinear Schr\"odinger equation \cite{KiOhPoVi14}.
\end{rema}

\subsection{Trapping and global well-posedness}
In this section we prove Theorem \ref{T:global}.
We use a classical trapping argument, that goes back to \cite{PaSa75} in the context of the Klein-Gordon equation (see e.g. \cite{Stu91} for NLS). We start by proving the following stationary lemma:


\begin{lemma}\label{L:control}
Assume $p>1$, $\lambda<\frac{(n-1)^2}4$,  and $1<p<1+\frac 4{n-2}$ if $n\geq 3$. Then if $E_\lambda(f)\leq E_\lambda(Q_{\lambda})$ and $\|f\|_{{\HHH}_\lambda}^2\leq \|Q_{\lambda}\|^2_{{\HHH}_\lambda}$ we have
\begin{equation}
\label{var}
\|f\|^2_{{\HHH}_\lambda}\leq \frac{E_\lambda(f)}{E_\lambda(Q_{\lambda})}\|Q_{\lambda}\|^2_{{\HHH}_\lambda}.
\end{equation}
In particular there is no function $f$ such that $E_\lambda(f)< E_\lambda(Q_{\lambda})$ and $\|f\|_{{\HHH}_\lambda}^2= \|Q_{\lambda}\|^2_{{\HHH}_\lambda}.$
\end{lemma}

\begin{proof}
Recall from subsection \ref{SS:Ground_states} the variational definition of $D_{\lambda}$. By \eqref{formule_Qp}, \eqref{formule_Dp} 
\begin{gather}
\notag
\|Q_{\lambda}\|_{{\HHH}_\lambda}^2=\|Q_{\lambda}\|_{L^{p+1}}^{p+1}\\
 \label{Qp2}
D_{\lambda}=\|Q_{\lambda}\|_{{\HHH}_\lambda}^{\frac{2(1-p)}{p+1}}\\
\label{Qp3}
E_\lambda(Q_{\lambda})=\|Q_{\lambda}\|_{{\HHH}_\lambda}^2\frac{p-1}{2(p+1)}.
\end{gather} 
Therefore
\begin{multline}
\label{bound_Em}
 E_\lambda(f)=\frac{1}{2}\|f\|^2_{{\HHH}_\lambda}-\frac{1}{p+1}\|f\|^{p+1}_{L^{p+1}}
\geq \frac{1}{2}\|f\|^2_{{\HHH}_\lambda}-\frac{D_{\lambda}^\frac{p+1}2}{p+1}\|f\|^{p+1}_{{\HHH}_\lambda}= a\left(\|f\|^2_{{\HHH}_\lambda}\right),
\end{multline}
where (in view of \eqref{Qp2})
$$ a(x)=\frac{1}{2}x-\frac{\|Q_{\lambda}\|^{1-p}_{{\HHH}_\lambda}}{p+1}x^{\frac{p+1}{2}}.$$
In particular,
$$b(x)=a(x)-\frac{p-1}{2(p+1)}x$$ 
vanishes at $x=0$ and at $x=\|Q_{\lambda}\|_{{\HHH}_\lambda}^2$, increases on $\left[0,\left(\frac 2{p+1}\right)^\frac 2{p-1}\|Q_{\lambda}\|_{{\HHH}_\lambda}^2\right]$ and decreases on $\left[\left(\frac 2{p+1}\right)^\frac 2{p-1}\|Q_{\lambda}\|_{{\HHH}_\lambda}^2,\|Q_{\lambda}\|_{{\HHH}_\lambda}^2\right]$ so it is a positive function on the whole interval $[0,\|Q_{\lambda}\|_{{\HHH}_\lambda}^2]$. Since $\|f\|^2_{{\HHH}_\lambda}\leq \|Q_{\lambda}\|^2_{{\HHH}_\lambda}$, combining with \eqref{bound_Em} we have obtained that
$$E_\lambda(f)\geq  a\left(\|f\|^2_{{\HHH}_\lambda}\right)\geq \frac{p-1}{2(p+1)}\|f\|^2_{{\HHH}_\lambda}.$$
Dividing this estimate by the value \eqref{Qp3} of $E_\lambda(Q_{\lambda})$ we obtain \eqref{var}.
\end{proof}

\begin{proof}[Proof of Theorem \ref{T:global}]

Let $u_0$ be as in Theorem \ref{T:global}.

If $\delta_\lambda(u_0)=0$, then by \eqref{var}, $E_{\lambda}(u_0)=E_{\lambda}(Q_{\lambda})$. Thus $u_0$ is a minimizer for Poincar\'e-Sobolev inequality and by Theorem \ref{T:minimizers},
$$u_0(x)=e^{i\theta}Q(h(x)),$$
for some $Q\in \QQQ_{\lambda}$, $\theta\in \RR$, and $h\in \GG$, which gives Case \eqref{I:zero}.

As a consequence of Case \eqref{I:zero}, if $\delta_{\lambda}(u_0)\neq 0$, then $\delta_{\lambda}(u(t))\neq 0$ for all $t$ in the domain of existence of $u$, which proves by continuity that $\delta_{\lambda}(u(t))$ does not change sign.

We next assume that $\delta_{\lambda}(u_0)$ is  negative, and thus that $\delta_\lambda(u(t))$ is negative for all $t$. This ensures that the $\mathcal H$ norm of $u(t)$ is bounded in time. By mass conservation we deduce that the $H^1$ norm of $u$ is bounded 
and global well-posedness follows from the blow-up criterion \eqref{blupcrit} mentioned in the introduction. This proves case \eqref{I:negative}.

In case \eqref{I:positive}\, $\delta_\lambda(u(t))>0$ for all $t$ in the domain of existence of $u$ and thus 
$$E_\lambda(u_0)\leq E_\lambda(Q_{\lambda})=\frac 12\|Q_{\lambda}\|^2_{{\HHH}_\lambda}-\frac{1}{p+1}\|Q_{\lambda}\|^{p+1}_{L^{p+1}}<\frac 12\|u(t)\|^2_{{\HHH}_\lambda}-\frac{1}{p+1}\|Q_{\lambda}\|^{p+1}_{L^{p+1}}.$$ 
If the solution $u$ scatters for positive times, then 
$$\lim_{t\to\infty}\|u(t)\|_{L^{p+1}}=0,$$
and for any $\epsilon>0$ there exists $t$ large such that 
$$\frac 12\|u(t)\|^2_{{\HHH}_\lambda}<E_\lambda(u(t))+\epsilon=E_\lambda(u_0)+\epsilon,$$
so we get a contradiction by taking $\epsilon =\frac{1}{2(p+1)}\|Q_{\lambda}\|^{p+1}_{L^{p+1}}$.
\end{proof}
\section{Construction of the critical solution}
\label{S:critical}
\subsection{Some spaces of functions and inequalities}
\label{SS:spaces}

\subsubsection{Preliminaries on Fourier analysis on hyperbolic space}
\label{SS:preliminaries}
We will mostly use the notations of \cite{IoPaSt12}. We refer to this article for more details.

We define the Fourier transform on $\HH^n$, following the general definition of the Fourier transform on symmetric spaces given in \cite{He65}. For $\omega\in S^{n-1}$, $\lambda\in \RR$, we 
define the Fourier transform of $f\in L^1(\HH^n)$ by
$$\hat{f}(\lambda,\omega)=\int_{\HH^n} f(x)[x,(1,\omega)]^{i\lambda-\rho}d\mu(x),\quad \rho=\frac{n-1}{2}.$$
We have
\begin{equation*}
 \widehat{\Delta_{\HH^n}f}(\lambda,\omega)=-\left(\lambda^2+\rho^2\right)\tilde{f}(\lambda,\omega).
\end{equation*} 
The Fourier inversion formula reads
\begin{equation}
\label{Fourier_inversion}
 f(x)=\int_{-\infty}^{\infty} \int_{S^{n-1}}\hat{f}(\lambda,\omega)[x,(1,\omega)]^{-i\lambda-\rho}|c(\lambda)|^{-2}d\lambda d\omega,
\end{equation} 
where the Harish-Chandra function $c(\lambda)$ is defined by 
$$|c(\lambda)|^{-2}=\frac 12\frac{|\Gamma(\rho)|^2}{|\Gamma(2\rho)|^2}\frac{|\Gamma(\rho+i\lambda)|^2}{|\Gamma(i\lambda)|^2}.$$
We note that $|c(\lambda)|^{-2}$ is of the order $\lambda^{n-1}$ as $\lambda\to\infty$, and $\lambda^{2}$ as $\lambda\to 0$.

A version of Plancherel theorem is also available on $\HH^n$: the Fourier transform $f\mapsto \hat{f}$ extends to an isometry of $L^{2}(\HH^n)$ onto $L^2\left((-\infty,\infty)\times S^{n-1}),|c(\lambda)|^{-2}d\lambda d\omega\right)$, and, for $f,g\in L^{2}(\HH^n)$, 
$$\int_{\HH^n}f(x)\overline{g}(x)d\mu=\int_{-\infty}^{\infty}\int_{S^{n-1}} \hat{f}(\lambda,\omega)\overline{\hat{g}}(\lambda,\omega)|c(\lambda)|^{-2}d\lambda d\omega.$$
We will use the spectral projectors $P_m$, $m>0$ defined as follows
\begin{equation}
 \label{def_Pm}
 P_m=-\frac{1}{m^2}\Delta e^{\frac{1}{m^2}\Delta},
\end{equation} 
that is 
$$\widehat{P_mf}(\lambda,\omega)=\frac{1}{m^2}\left(\lambda^2+\rho^2\right)e^{-\frac{\lambda^2+\rho^2}{m^2}}\hat{f}(\lambda,\omega).$$
For $s\in \RR$, we define the Sobolev space $H^s(\HH^n)$ as the closure of $C^{\infty}_0(\HH^n)$ for the norm 
$$\|f\|_{H^s}=\left\|(-\Delta)^{s/2}f\right\|_{L^2}.$$
Note that 
$$\|f\|_{H^s}^2\approx \int_{-\infty}^{\infty}\int_{\HH^n} \left(\rho^2+\lambda^2\right)^{s}\left|\hat{f}(\lambda,\omega)\right|^2|c(\lambda)|^{-2}d\lambda d\omega.$$


\subsubsection{A refined subcritical Sobolev inequality}
Recall from \eqref{def_Pm} the definition of the spectral projector $P_m$. For $s\in (0,n/2]$, we define the Banach space $B^s$ as the closure of $C^{\infty}_0(\HH^n)$  
for the $B^{-(\frac n2-s),\infty}_\infty$ Besov-type norm: 
\begin{equation*}
\|u\|_{B^s}=\sup_{m\geq 1} m^{s-n/2} \|P_mf\|_{L^\infty(\mathbb H^n)}.
\end{equation*} 
\begin{lemma}\label{lemmaPmest}
For $0<s\leq  n/2$, there exists $C>0$ such that for all $f\in H^s$, one has $f\in B^s$ and
\begin{gather}
\label{Bs_Hs}
 \|f\|_{B^s}\leq C\|f\|_{H^s}\\
\label{Hs_all_m}
\forall m>0,\quad |P_mf(x)|\leq C\left(\frac{1}{m^2}+m^{\frac{n}{2}-s}\right)e^{-\frac{\rho^2}{m^2}}\|f\|_{H^s}.
\end{gather} 
\end{lemma}
\begin{proof}
 Let $f\in H^s$. By the definition of $P_m$ and Fourier inversion formula \eqref{Fourier_inversion},
\begin{multline}
\label{BsHs1}
 |P_mf(x)|\leq \int_{-\infty}^{\infty}\int_{S^{n-1}} \frac{1}{m^2} \left(\lambda^2+\rho^2\right)e^{-\frac{\lambda^2+\rho^2}{m^2}}\left|\hat{f}(\lambda,\omega)\right|\,|c(\lambda)|^{-2}[x,(1,\omega)]^{-\rho}\,d\lambda\, d\omega\\
 \leq \sqrt{\int_{-\infty}^{\infty}\int_{S^{n-1}} \left(\lambda^2+\rho^2\right)^s \left|\hat{f}(\lambda,\omega)\right|^2 |c(\lambda)|^{-2}d\lambda\, d\omega}\\ \times\sqrt{2\int_{0}^{\infty}\frac{1}{m^4} \left(\lambda^2+\rho^2\right)^{2-s}e^{-\frac{2(\lambda^2+\rho^2)}{m^2}}|c(\lambda)|^{-2}  d\lambda\int_{S^{n-1}}[x,(1,\omega)]^{-2\rho}\,d\omega}.
\end{multline}
Using that $|c(\lambda)|^{-2}\sim \lambda^{2}$ as $\lambda\to 0$, we obtain
\begin{equation}
 \label{BsHs2}
\int_0^{1}\frac{\left(\lambda^2+\rho^2\right)^{2-s}e^{-\frac{2(\lambda^2+\rho^2)}{m^2}}}{m^4|c(\lambda)|^{2}}d\lambda 
\leq \frac{e^{-\frac{2\rho^2}{m^2}}}{m^4}\int_0^1\lambda^2\left( \lambda^2+\rho^2 \right)^{2-s}\,d\lambda\leq C\frac{e^{-\frac{2\rho^2}{m^2}}}{m^4},
\end{equation} 
Furthermore, using $|c(\lambda)|^{-2}\sim \lambda^{n-1}$ as $\lambda\to\infty$
\begin{multline}
\label{BsHs3}
\int_1^{\infty}\frac{\left(\lambda^2+\rho^2\right)^{2-s}e^{-\frac{2(\lambda^2+\rho^2)}{m^2}}}{m^4|c(\lambda)|^{2}}d\lambda\leq Ce^{-\frac{2\rho^2}{m^2}} \int_1^{\infty} e^{-\frac{2\lambda^2}{m^2}}\lambda^{4-2s} \frac{\lambda^{n-1}}{m^4}\,d\lambda\\
\leq Ce^{-\frac{2\rho^2}{m^2}}m^{n-2s} \int_{1/m}^{\infty} e^{-2\sigma^2}\sigma^{3-2s+n}d\sigma \leq Cm^{n-2s}e^{-\frac{2\rho^2}{m^2}}.
\end{multline}
Finally, we claim that the spherical function-like integral
$$\int_{S^{n-1}}[x,(1,\omega)]^{-2\rho}\,d\omega=C\int_0^\pi (\cosh |x|-\sinh |x|\cos \alpha)^{-2\rho}\sin^{n-2}\alpha\,d\alpha,$$
is uniformly bounded in $|x|$. Indeed, we have (for some constant $c_n>0$)
 \begin{align*}
  F(r)&=c_n\int_0^{\pi} \left(\cosh r-\cos\theta\sinh r\right)^{1-n} \sin^{n-2}\theta\,d\theta\\
 &= c_n\int_0^{\pi} \left(\cosh r-\sinh r+(1-\cos\theta)\sinh r\right)^{1-n} \sin^{n-2}\theta\,d\theta\\
 &= c_n \int_0^{\pi} e^{(n-1)r} \left(1+e^{r}\sinh r(1-\cos\theta)\right)^{1-n}\sin^{n-2}\theta\,d\theta.
 \end{align*}
We have $\max_{r\in [0,1]}F(r)<\infty$. Assuming $r\geq 1$, we obtain:
\begin{align*}
 F(r)&\leq c_n\int_0^{\pi} e^{(n-1)r} \left(1+C^{-1} e^{2r}\theta^2\right)^{1-n}\theta^{n-2}\,d\theta\\
 &\leq c_n \int_0^{e^r\pi} \left(1+\frac{\sigma^2}{C}\right)^{1-n} \sigma^{n-2}\,d\sigma,
\end{align*}
by the change of variable $\sigma=e^r\theta$. This concludes the proof since
$$2(1-n)+n-2=-n\leq -2.$$
Combining this with \eqref{BsHs1}, \eqref{BsHs2} and \eqref{BsHs3}, we obtain \eqref{Hs_all_m}. Inequality \eqref{Bs_Hs} follows. 
\end{proof}

We next prove a refined Sobolev inequality which generalizes \cite[Lemma 2.2,ii)]{IoPaSt12} which treats the case $s=1$, $n=3$. It is in the spirit of the refined Sobolev embedding on Euclidean space in \cite{GeMeOr96}. More precisely, the inequality we prove is the analog on the hyperbolic space of the inequality
$$ \|f\|_{L^{\alpha}(\mathbb R^n)}\leq C\|f\|_{H^{\frac n2-\frac n{\alpha}}(\mathbb
R^n)}^\frac 2\alpha \|f\|_{B^{-\frac n{\alpha},\infty}_\infty(\mathbb
R^n)}^{1-\frac 2\alpha}, \quad 2<\alpha<\infty.$$
\begin{prop}
 \label{P:Sobolev}
 Let $0<s<\min\{2,n/2\}$ and $\alpha$ such that $\frac{1}{\alpha}=\frac{1}{2}-\frac sn$. There is a constant $C>0$ such that for all $f\in H^s$,
 \begin{equation}
  \label{A1}
  \|f\|_{L^{\alpha}}\leq C\|f\|_{H^s}^{\frac{2}{\alpha}}\|f\|^{1-\frac{2}{\alpha}}_{B^s}.
 \end{equation} 
\end{prop}
\begin{proof}
 We use a method based on spectral calculus that goes back to \cite{ChXu97}.
 By the definition \eqref{def_Pm} of $P_m$
 \begin{equation}
 \label{A2}
  \int_0^{A}\frac{1}{m}\widehat{P_m(f)}\,dm=\int_0^A\frac{1}{m^3}(\lambda^2+\rho^2)e^{-\frac{\lambda^2+\rho^2}{m^2}}\hat f\,dm=\frac 12 e^{-\frac{\lambda^2+\rho^2}{A^2}}\hat f,
 \end{equation} 
 for any $A>0$.
 
 \emph{Step 1. Low-frequency bound.} Here we prove the desired estimate for  the low frequencies part $e^{\Delta}f= \int_0^{1}\frac{1}{m}P_m(f)\,dm$. Since $\|e^{t\Delta}f\|_{L^\infty}\leq \|f\|_{L^\infty}$  (see for instance \cite{GrNo98}) and $\|e^{t\Delta}f\|_{L^2}\leq e^{-\rho^2t}\|f\|_{L^2}$ (this follows from $(-\Delta f,f)_{L^2}\geq \rho^2\|f\|^2_{L^2}$), we get 
$$\|e^{t\Delta}f\|_{L^\alpha}\leq e^{-ct}\|f\|_{L^\alpha},$$ 
where $c=\frac{2\rho^2}{\alpha}$. Then 
 $$\|e^{\Delta}f\|_{L^\alpha}\leq C\int_0^{1}\frac{1}{m}\|P_m(f)\|_{L^\alpha}\,dm\leq C\int_0^{1}\frac{1}{m^3}\left\|e^{(\frac 1{m^2}-\frac 12)\Delta}P_{\sqrt{2}}(f)\right\|_{L^\alpha}\,dm$$
 $$\leq C\int_0^{1}\frac{e^{c(\frac{1}{2}-\frac 1{m^2})}}{m^3}\|P_{\sqrt{2}}(f)\|_{L^\alpha}\,dm\leq C\|P_{\sqrt{2}}(f)\|_{L^\alpha}\leq C\|P_{\sqrt{2}}(f)\|_{L^2}^{\frac 2\alpha}\|P_{\sqrt{2}}(f)\|_{L^\infty}^{1-\frac 2\alpha},$$
 and we conclude \eqref{A1} by noting that since we are at low frequencies, $\|P_{\sqrt{2}}(f)\|_{L^2}\leq C\|f\|_{H^s}$.

Now we shall treat the high frequencies. We let 
 \begin{equation}
  \label{def_g}
 g=2\int_1^{\infty}\frac{1}{m}P_m(f)\,dm=f-e^{\Delta}f 
 \end{equation} 
 in view of \eqref{A2}.
To complete the proof of the proposition, in view of Step 1, we need to prove 
$$\|g\|_{L^{\alpha}}\leq C\|f\|^{\frac{2}{\alpha}}_{H^s}\|f\|_{B^s}^{1-\frac{2}{\alpha}},$$
which we will do in two steps. We first introduce some notations.

Let, for $R>0$, 
 $$ A_R=\left(\frac{R\left(\frac n2-s\right)}{4\|f\|_{B^s}}\right)^{\frac{2}{n-2s}}.$$
We write 
$$g=g_{\leq A_R}+g_{>A_R},$$
where, if $A_R\geq 1$, 
$$g_{\leq A_R}=2\int_1^{A_R}\frac{1}{m}P_m(f)\,dm, \quad g_{>A_R}=2\int_{A_R}^{\infty}\frac{1}{m}P_m(f)\,dm.$$
and if $A_R<1$, $g_{\leq A_R}=0$, $g_{>A_R}=g$.

 \emph{Step 2.} In this step, we prove:
 \begin{equation}
  \label{A3}
  \|g\|_{L^{\alpha}}^{\alpha}\leq C\int_0^{\infty} R^{\alpha-3}\left\|g_{>A_R}\right\|_{L^2}^2\,dR.
 \end{equation} 
 Indeed, if $A_R\geq 1$, by the definition of $B^s$ and $g_{\leq A_R}$
 \begin{equation*}
  \left|g_{\leq A_R}\right|\leq 2\int_0^{A_R} \frac{1}{m}m^{\frac{n}{2}-s}\|f\|_{B^s}\,dm=\frac{4A_R^{\frac{n}{2}-s}}{n-2s}\|f\|_{B^s}=\frac{R}{2}.
 \end{equation*}
 This inequality remains valid if $A_R<1$ since the left-hand side is zero.
 Thus
$$\mu\left(\{|g|>R\}\right)\leq \mu\left(\{|g_{> A_R}|>R/2\}\right)\leq \frac{4}{R^2}\left\|g_{> A_R}\right\|_{L^2}^2.$$

$$\left\|g\right\|^{\alpha}_{L^{\alpha}}=\alpha \int_0^{\infty} R^{\alpha-1} \mu\left(\{|g|>R\}\right)\,dR\leq C\int_0^{\infty}R^{\alpha-3}\left\|g_{> A_R}\right\|^2_{L^2}\,dR.$$
 Hence \eqref{A3}.
 
 \emph{Step 3.}  By \eqref{A2} and the definition of $g_{>A_R}$,
 $$\tilde{g}_{> A_R}(\lambda,\omega)=\left(1-e^{-\frac{\lambda^2+\rho^2}{B_R^2}}\right)\tilde{f}(\lambda,\omega),$$
 where $B_R=\max(A_R,1)$.
 By \eqref{A3},
 \begin{multline}
  \label{A3'}
  \|g\|^{\alpha}_{L^{\alpha}}\leq C\int_0^{\infty} R^{\alpha-3} \int_0^{\infty} \int_{S^{n-1}} \left(1-e^{-\frac{\lambda^2+\rho^2}{B_R^2}}\right)^2\left|\hat{f}(\lambda,\omega)\right|^2|c(\lambda)|^{-2}d\omega\,d\lambda\,dR\\
  =C \int_0^{\infty} \int_{S^{n-1}}\left|\hat{f}(\lambda,\omega)\right|^2|c(\lambda)|^{-2} \int_0^{\infty} R^{\alpha-3}\left(1-e^{-\frac{\lambda^2+\rho^2}{A_R^2}}\right)^2\,d R\,d\omega\,d\lambda.
 \end{multline}
By the definition of $A_R$ and the change of variable in $R$
$$r=\frac{R(\frac{n}{2}-s)}{4\|f\|_{B^s}(\lambda^2+\rho^2)^{\frac{n-2s}{4}}}=\left(\frac{A_R^2}{\lambda^2+\rho^2}\right)^{\frac{n-2s}{4}},$$ 
we deduce
\begin{multline}
 \label{A4}
\int_0^{\infty} R^{\alpha-3}\left(1-e^{-\frac{\lambda^2+\rho^2}{A_R^2}}\right)^2\,dR\\
\leq C\|f\|_{B^s}^{\alpha-2}(\lambda^2+\rho^2)^{\frac{n-2s}{4}(\alpha-2)}\int_0^{\infty} r^{\alpha-3}\left(1-e^{-r^{-\frac{4}{n-2s}}}\right)^2\,dr.
\end{multline} 
Note that $\alpha-3>-1$ and, as $r$ goes to infinity,
$$r^{\alpha-3}\left(1-e^{-r^{-\frac{4}{n-2s}}}\right)^2\approx r^{\frac{4s-8}{n-2s}-1},$$
which proves (using that $s<\frac n2$ and $s<2$) that the integral at the right-hand side of \eqref{A4} is finite. Going back to \eqref{A3'}, we obtain
\begin{equation*}
 \|g\|^\alpha_{L^\alpha}\leq C\|f\|^{\alpha-2}_{B^s}\int_0^{\infty} \int_{S^{n-1}} (\lambda^2+\rho^2)^{s} |\hat{f}(\lambda,\omega)|^2|c(\lambda)|^{-2}d\omega\,d\lambda=C\|f\|^{\alpha-2}_{B^s}\|f\|^2_{H^s}
\end{equation*} 
which concludes the proof.
\end{proof}

\subsubsection{An interpolation inequality}
\label{SSS:interpol}
\begin{prop}
\label{P:inter}
 There exists $\theta\in (0,1)$ and a constant $C>0$, both depending on $p$, such that 
\begin{equation}
\label{inter_1}
\|e^{it\Delta}f\|_{S^0(\RR)}\leq C\left\|e^{it\Delta} f\right\|^{\theta}_{L^{\infty}_tB^s}\|f\|^{1-\theta}_{H^s},
\end{equation} 
where $s=\frac{n}{2}-\frac{n}{p+1}\in (0,1)$.
\end{prop}
\begin{proof}
\emph{First case: $1<p\leq \frac{4}{n}+1$.} 

In this case, $S^0(\RR)=L^{p+1}(\RR\times \HH^n)$. By the refined Sobolev inequality \eqref{A1}, 
$$\|e^{it\Delta}f\|_{L^{\infty}_tL^{p+1}_x}\leq C\|e^{it\Delta}f\|_{L^{\infty}_tH^s}^{\frac{2}{p+1}}\|e^{it\Delta} f\|_{L^{\infty}_tB^s}^{1-\frac{2}{p+1}},$$
where by definition, $s=\frac n2-\frac{n}{p+1}\in (0,(1+2/n)^{-1})$. Hence
\begin{equation}
 \label{inter_2}
\|e^{it\Delta}f\|_{L^{\infty}_tL^{p+1}_x}\leq C\|f\|^{\frac{2}{p+1}}_{H^s}\|e^{it\Delta}f\|^{1-\frac{2}{p+1}}_{L^{\infty}_tB^s}.
\end{equation} 
Moreover, by the Strichartz inequalities on $\HH^n$ (see Theorem \ref{T:Strichartz}), for all $\gamma$ with $2<\gamma<2+\frac{4}{n}$,
\begin{equation}
 \label{inter_3}
\left\| e^{it\Delta}f\right\|_{L^{\gamma}_tL^{\gamma}_x\cap L^{\gamma}_tL^{\beta}_x}\leq C\|f\|_{L^2},
\end{equation} 
where $\beta=\frac{2n\gamma}{n\gamma-4}$. Note that
$$ \lim_{\gamma\to 2}\beta=\begin{cases} +\infty& \text{ if }n=2\\ \frac{2n}{n-2}&\text{ if }n\geq 3\end{cases}.$$
Choosing $\gamma>2$ close enough to $2$, we obtain $2<\gamma<p+1<\beta$. For these value of $\gamma$, \eqref{inter_3} implies
\begin{equation}
 \label{inter_4}
\left\| e^{it\Delta}f\right\|_{L^{\gamma}_tL^{p+1}_x}\leq C\|f\|_{L^2}.
\end{equation} 
Combining \eqref{inter_2}, \eqref{inter_4}, and H\"older's inequality we obtain \eqref{inter_1} with $\theta=\left(1-\frac{2}{p+1}\right)\left(1-\frac{\gamma}{p+1}\right)$.

\emph{Second case: $\frac{4}{n}+1<p$ and $p<1+\frac{4}{n-2}$ if $n\geq 3$.} In this case 
$$ S^0(\RR)=L^{a}(\RR,L^{p+1})\text{ where } a=\frac{2(p-1)(p+1)}{4-(n-2)(p-1)}.$$
By Strichartz estimates (Theorem \ref{T:Strichartz}),
\begin{equation}
 \label{inter_5} 
\|e^{it\Delta} f\|_{L^{\gamma}_tL^{p+1}_x}\leq C\|f\|_{L^2},
\end{equation} 
where $\gamma=\frac{4(p+1)}{n(p-1)}<a$ if $p>1+\frac{4}{n}$. By the generalized Sobolev inequality \eqref{A1}, with $\alpha=p+1$, $s=\frac{n}{2}-\frac{n}{p+1}\in (0,1)$, 
$$\forall t,\quad \|e^{it\Delta}f\|_{L^{p+1}}\leq C\|e^{it\Delta} f\|_{H^s}^{\frac{2}{p+1}}\|e^{it\Delta}f\|_{B^s}^{1-\frac{2}{p+1}}.$$
Hence
\begin{equation}
 \label{inter_6}
\|e^{it \Delta}f\|_{L^{\infty}_tL^{p+1}_x}\leq C\|f\|_{H^s}^{\frac{2}{p+1}}\|e^{it\Delta}f\|_{L^{\infty}_tB^s}^{1-\frac{2}{p+1}}.
\end{equation} 
Combining \eqref{inter_5} and \eqref{inter_6} and using $\gamma<a$, we obtain \eqref{inter_1} in this case also.
\end{proof}

\subsection{Profile decomposition}
\label{SS:profile}
\subsubsection{Linear profile decomposition}
Recall from \S \ref{SS:preliminaries} the definition of the isometry group $\GG$. We denote by $d$ the geodesic distance on $\HH^n$.
Define, for $f\in H^1$, 
$$ \|f\|_{\Sigma}=\sup_{\substack{m\geq 1, t\in \RR}} \frac{m^{1-n/2}}{\log(m+2)} \|P_me^{it\Delta} f\|_{L^\infty(\mathbb H^n)}.$$
In particular we have for all $s\in (0,1)$, for all $t\in \RR$,
\begin{equation}
\label{Bs-Sigma} 
\|e^{it\Delta}f\|_{B^s} \leq C_s\|f\|_{\Sigma}.
\end{equation}
By \eqref{Hs_all_m},
$$ \|f\|_{\Sigma}\leq C\|f\|_{H^1}.$$

\begin{prop}[Subcritical profile decomposition]
\label{P:profile}
 Let $(f_k)_{k}$ be a bounded sequence in $H^1(\HH^n)$. Then there exists a subsequence of $(f_k)_{k}$  (that we still denote by $(f_k)_{k}$), a family $(\varphi_j)_{j\geq 1}$ of functions in $H^1(\HH^n)$ and, for each $j\geq 1$, a sequence $\big((t_{j,k},h_{j,k})\big)_k$ in $\RR\times \mathbb G$ such that
\begin{gather}
 \label{A11} \sum_{j\geq 1} \|\varphi_j\|^2_{H^1}<\infty\\
\label{A12}
j\neq j'\Longrightarrow \lim_{k\to \infty} d(h_{j,k}\cdot 0,h_{j',k}\cdot 0) +|t_{j,k}-t_{j',k}|=+\infty\\
\label{A13}
\forall j\geq 1,\quad e^{-it_{j,k}\Delta} f_k(h_{j,k}^{-1}\cdot) \xrightharpoonup[k\to\infty]{} \varphi_j\text{ weakly in }H^1
\end{gather}
and, denoting by 
\begin{equation*}
r_{J,k}=f_k-\sum_{j=1}^J e^{it_{j,k}\Delta}\varphi_j(h_{j,k}\cdot)
\end{equation*} 
we have
\begin{equation}
 \label{A15}
\lim_{J\to \infty}\varlimsup_{k\to\infty} \|r_{J,k}\|_{\Sigma}+\|e^{it\Delta}r_{J,k}\|_{S^0(\RR)}=0.
\end{equation} 
\end{prop}
We refer for example to \cite{MeVe98}, \cite{Ke01} for profile decompositions for the Schr\"odinger equation on $\RR^n$.
The $H^1$-critical profile decomposition on the space $\HH^3$ was constructed in \cite{IoPaSt12} (see also \cite{LaSuSo14Pb} for the analogous result for the wave equation on $\HH^n$). In this setting, profiles might concentrate at one point of $\HH^n$, and become solutions of the Schr\"odinger equation on the Euclidean space. In our case, this is prevented by the subcriticality of the problem.
\begin{nota}
In what follows, we will often extract subsequences from a given sequence. To lighten notations, we will always, as in the preceding proposition, use the same notation for the extracted subsequence and the original sequence.
\end{nota}

\begin{rema}
\label{R:radial}
 It follows from \eqref{A13} that if the $f_k$ are all radial, then we can assume that $h_{j,k}$ is the identity of $\HH^n$ for all $j,k$, and that all profiles $\varphi_j$ are radial.
\end{rema}

\begin{defi}
If $\left(\varphi_j;(t_{j,k},h_{j,k})_{k}\right)_{j\geq 1}$ satisfies the conclusions of Proposition \ref{P:profile}, we say that it is a \emph{profile decomposition} for the sequence $(f_k)_{k}$. 
\end{defi}
We postpone the proof of the profile decomposition, and state the following Pythagorean expansions, to be proved at the end of this section:
\begin{prop}
 \label{L:Pyt}
 Let $\lambda<\frac{(n-1)^2}{4}$.
Let $(f_k)_k$ be a bounded sequence in $H^1$ that admits a profile decomposition $\left(\varphi_j;(t_{j,k},h_{j,k})_{k}\right)_{j\geq 1}$.
Then
\begin{gather}
 \label{A16} \forall J\geq 1,\quad \lim_{k\to\infty} \|f_k\|^2_{\HHH_{\lambda}}-\sum_{j=1}^J \|\varphi_j\|^2_{\HHH_{\lambda}} -\|r_{J,k}\|_{\HHH_{\lambda}}^2=0\\
\label{A17} \lim_{k\to\infty} \|f_k\|^{p+1}_{L^{p+1}}-\sum_{j=1}^{+\infty} \left\|e^{-it_{j,k}\Delta}\varphi_j\right\|^{p+1}_{L^{p+1}}=0.
\end{gather}
\end{prop}
To prove Proposition \ref{P:profile} and Proposition \ref{L:Pyt}, we need the following lemma:
\begin{lemma}
\label{C:ortho}
 Let $f,g\in H^1(\HH^n)$, $(t_k,h_k)_k$ and $(t'_k,h'_k)$ two sequences in $\RR\times \mathbb G$ such that
\begin{equation*}
\lim_{k\to\infty} d(h_k\cdot 0,h'_k\cdot 0)+|t_k-t_k'|=+\infty.
\end{equation*} 
Then
\begin{gather}
 \label{A19} 
\forall \lambda<\frac{(n-1)^2}{4},\quad 
\lim_{k\to\infty} \left(e^{it_k\Delta} f(h_k\cdot),e^{it_k'\Delta}g(h_k'\cdot)\right)_{\HHH_{\lambda}}=0\\
 \label{A20} 
\lim_{k\to\infty} \int \left|e^{it_k\Delta} f(h_kx)\right|\,\left|e^{it_k'\Delta}g(h_k'x)\right|^{p}\,dx=0.
\end{gather}
\end{lemma}
\begin{proof}
 \emph{Proof of \eqref{A19}.} By density we can assume, without loss of generality, $f,g\in C_0^{\infty}(\HH^n)$. We have
 \begin{multline*}
\left(e^{it_k\Delta}f(h_k\cdot),e^{it_k'\Delta}g(h_k'\cdot)\right)_{\HHH_{\lambda}}=-\left((\Delta+\lambda) e^{it_k\Delta}f(h_k\cdot),e^{it_k'\Delta}g(h_k'\cdot)\right)_{L^2}\\
=\left(-(\Delta+\lambda) f,e^{i(t_k'-t_k)\Delta}g(h_k'\circ h_k^{-1}\cdot)\right)_{L^2}.
 \end{multline*}
If $|t_k-t_k'|\to \infty$ as $k\to\infty$, then $\|e^{i(t_k'-t_k)\Delta}g\|_{L^{\infty}}\to 0$ and the result follows. If not, we 
 can assume without loss of generality:
 $$ \lim_{k\to\infty} t_k'-t_k=\theta\in\RR,$$
and \eqref{A19} is equivalent to
\begin{equation}
\label{limite_0}
\lim_{k\to\infty}\left((\Delta+\lambda) f,(e^{i\theta\Delta}g)(h_k'\circ h_k^{-1}\cdot)\right)_{L^2}=0.
 \end{equation} 
Furthermore
\begin{equation}
\label{lim_dhk}
\lim d(0,h_k'\circ h_k^{-1}\cdot 0)=+\infty. 
\end{equation} 
If $\theta=0$, the support of $(\Delta+\lambda) f$ and $(e^{i\theta\Delta}g)(h_k'\circ h_k^{-1}\cdot)$ are disjoint for large $k$ and \eqref{limite_0} follows. If not, one can approximate $e^{i\theta\Delta}g$, in $L^2$, by compactly supported functions which yields \eqref{limite_0}, arguing again on the supports.

\emph{Proof of \eqref{A20}.} Note that by Sobolev embeddings and conservation of the $H^1$-norm for the linear equation, the sequences 
$$\left(\|e^{it_k\Delta} f\|_{L^{p+1}}\right)_k\text{ and }\left(\|e^{it_k'\Delta} g\|_{L^{p+1}}\right)_k$$
are bounded. Furthermore, 
$$ \lim_{k\to\infty}t_k=\pm\infty\Rightarrow \lim_{k\to\infty}\left\|e^{it_k\Delta}f\right\|_{L^{p+1}}=0\text{ and }\lim_{k\to\infty}t_k'=\pm\infty\Rightarrow \lim_{k\to\infty}\left\|e^{it_k'\Delta}g\right\|_{L^{p+1}}=0.$$
In both cases, \eqref{A20} holds. Arguing by contradiction and extracting subsequences, we are reduced to prove \eqref{A20} when $t_k$ and $t_k'$ have finite limits as $k$ goes to infinity. Time translating, we can also assume that these limits are both $0$, and we see that it is sufficient to prove:
$$ \forall f,g\in H^1(\HH^n),\quad \lim_{k\to\infty}\int |f(h_k\cdot x)||g(h_k'\cdot x)|^p\,dx=0,$$
provided \eqref{lim_dhk} holds. This follows by approximating $f$ and $g$, in $L^{p+1}$, by compactly supported functions and arguing on the supports.
\end{proof}
We next prove Proposition \ref{P:profile}. 
\begin{proof}
 We shall use the following general abstract concentration-compactness result (see Proposition 2.1, Definition 2.2 and Theorem 2.3 of \cite{ScTi02}): 
 \begin{exttheo}[\cite{ScTi02}] Let $H$ be a separable Hilbert space and $D$ a group of unitary operators in $H$ such that if $(g_k)_k\in D^{\NN}$ does not converge weakly to zero, then there exists a strongly convergent subsequence of $g_k$ such that s-$\lim_k g_k\neq 0$. 
 
 If $(f_k)_k\in H^{\NN}$ is a bounded sequence, then (extracting subsequences in $k$), there exist $\varphi_j\in H$, $(g_{j,k})_k\in D^{\NN}$, $j\geq 1$ such that 
  \begin{equation}
   \label{ST1}\sum_{j\geq 1}\|\varphi_j\|^2\leq \lim\sup_{k\to\infty}\|f_k\|^2,
     \end{equation} 
  \begin{equation}
   \label{ST2}(g_{j,k})^{-1}g_{j',k}\xrightharpoonup[k\to\infty]{} 0 \mbox{ for }j\neq j',
     \end{equation} 
  \begin{equation}
   \label{ST3}(g_{j,k})^{-1}f_k\xrightharpoonup[k\to\infty]{} \varphi_j,  
   \end{equation} 
 and for all $\phi\in H$,
  \begin{equation}
   \label{ST4}\lim_{J\to\infty}\varlimsup_{k\to\infty}\;\sup_{g\in D}\bigg|\bigg(g\Big(f_k-\sum_{j=1}^Jg_{j,k}\varphi_j\Big),\phi\bigg)\bigg|=0.
     \end{equation} 
\end{exttheo}
Note that in \cite[Theorem 2.3]{ScTi02}, \eqref{ST4} is stated without parameter $J$ and with an infinite sum. However \eqref{ST4} follows easily from the proof in \cite{ScTi02}.

We apply this result for $H=H^1(\mathbb H^n)$ and 
$$D=\{g:H^1(\mathbb H^n)\rightarrow H^1(\mathbb H^n)\;,\; g(f)(x)=e^{it\Delta} f(h\cdot x), (t,h)\in\RR\times \mathbb G\}.$$ 
 The hypothesis on $D$ is satisfied in view of Lemma \ref{C:ortho}. Indeed, if  $g_k=(t_k,h_k)_{k}\in D$ does not converge weakly to zero, then by taking $g_k'=(0,Id)\in D$, the conclusion \eqref{A19} ensures that $d(h_k\cdot 0, 0)+|t_k|$ does not tend to $+\infty$. Therefore, in view of the definition of $\mathbb G$, there exists strongly convergent subsequences of $(t_k)_k$ and $(h_k)_k$ in $\RR$ and $\GG$. This implies that $g_k$ has a strong limit $(\theta, h)\in D$. 
 
 Now we transcribe the results of this theorem in our context. 
 The statements \eqref{ST1} and \eqref{ST3} imply directly \eqref{A11} and \eqref{A13}. If the conclusion of \eqref{A12} does not hold, than by the same argument used above to check the assumption on $D$, we obtain a contradiction with \eqref{ST2}. Therefore \eqref{A12} is satisfied. 

We are left with proving \eqref{A15}. By the interpolation inequality \eqref{inter_1}, it is sufficient to consider only the $\|\cdot\|_{\Sigma}$ norm in \eqref{A15}. 

%
%
We will deduce \eqref{A15} from the following lemma, proved below.
\begin{lemma}
  \label{L:weak_CV}
  Let $(r_k)_{k}$ be a bounded sequence in $H^1(\HH^n)$. Assume that for all sequence $(t_k,h_k)_{k}$ in $\RR\times \mathbb G$, 
  \begin{equation}
   \label{weaklim_0}
   e^{it_k\Delta}r_k(h_k\cdot)\xrightharpoonup[k\to\infty]{}0\,\text{ weakly in }H^1(\HH^n).
  \end{equation} 
  Then
  \begin{equation}
   \label{stronglim_0}
 \lim_{k\to\infty} \|r_k\|_{\Sigma}=0.  
  \end{equation} 
 \end{lemma}
Assuming Lemma \ref{L:weak_CV}, we prove \eqref{A15} by contradiction. If \eqref{A15} does not hold, there exists $\eps>0$ and a sequence of positive integers $(J_{\ell})_{\ell\geq 0}$ such that
\begin{equation}
 \label{def_Jl}
 \lim_{\ell\to\infty}J_{\ell}=+\infty\quad \text{and}\quad\forall \ell\geq 0,\quad \varlimsup_{k\to\infty}\|r_{J_{\ell},k}\|_{\Sigma}\geq \eps.
\end{equation} 
Let $(\phi_{\alpha})_{\alpha\in \NN}$ be a countable, dense family in $H^1(\HH^n)$. Let $\nu\in \NN$. By \eqref{ST4} and \eqref{def_Jl}, there exists indexes $\ell(\nu)$ and $k({\nu})$ with the following properties
\begin{gather*}
 \forall \alpha\in \{0,\ldots,\nu\},\quad \sup_{\substack{t\in \RR\\ h\in G}} \left|\Big(e^{it\Delta}r_{J_{\ell(\nu)},k(\nu)}\left(h\cdot\right),\phi_{\alpha}\Big)_{H^1}\right|\leq \frac{1}{2^{\nu}}\\
 \|r_{J_{\ell(\nu)},k(\nu)}\|_{\Sigma}\geq \frac{\eps}{2}.
\end{gather*}
As a consequence of the first inequality and the density of $(\phi_{\alpha})_{\alpha\in \NN}$ in $H^1$, we obtain that for all sequence $(t_{\nu},h_{\nu})_{\nu}$ in $\RR\times \GG$,
$$ e^{it_{\nu}\Delta}r_{J_{\ell(\nu)},k(\nu)}(h_{\nu}\cdot)\xrightharpoonup[\nu\to\infty]{}0\,\text{ weakly in }H^1(\HH^n).$$
This proves that $(r_{J_{\ell(\nu)},k(\nu)})_{\nu}$ contradicts Lemma \ref{L:weak_CV}, concluding the proof.
 \end{proof}
 \begin{proof}[Proof of Lemma \ref{L:weak_CV}]
To prove Lemma \ref{L:weak_CV}, we argue by contradiction. Assume that \eqref{stronglim_0} does not hold. Then there exist a subsequence of $(r_k)_k$ (still denoted by $(r_k)_k$), $\eps>0$, a sequence $(t_k,x_k,m_k)_k$ in $\RR\times \HH^n\times [1,\infty)$ such that
\begin{equation}
 \label{contra_lim1}
 \forall k,\quad \left|P_{m_k}e^{it_k\Delta}r_k(x_k)\right|\frac{m_k^{1-n/2}}{\log(2+m_k)}\geq \eps.
\end{equation} 
Combining with \eqref{Hs_all_m} for $s=1$, we deduce
\begin{equation*}
 \forall k,\quad \eps\log(2+m_k)\leq C\|r_k\|_{H^1},
\end{equation*} 
which proves that the sequence $(m_k)_k$ is bounded. Extracting subsequences, we can assume
\begin{equation}
 \label{lim_mk}
 \lim_{k\to\infty}m_k=m\in[1,+\infty).
\end{equation} 
Since $\GG$ acts transitively on $\HH^n$, we can choose for all $k$ an isometry $h_k\in \GG$ such that $h_k(0)=x_k$. Let
$$ g_k(x)=e^{it_k\Delta} r_k(h_k\cdot x).$$
By \eqref{contra_lim1} and \eqref{lim_mk}, and since $h_k$ commutes with $P_m$, there exists $\eps'>0$ such that for large $k$
\begin{equation}
 \label{A21}
 \left|P_mg_k(0)\right|\geq \eps'.
\end{equation} 
By assumption \eqref{weaklim_0}, 
\begin{equation}
 \label{A22}
 g_k\xrightharpoonup[k\to\infty]{} 0 \text{ weakly in }H^1.
\end{equation} 
It follows from the inequality \eqref{Hs_all_m} 
that $f\mapsto P_m(f)(0)$ is a continuous linear form on $H^1(\HH^n)$, which combined with \eqref{A21} and \eqref{A22} yields a contradiction. The proof of Lemma \ref{L:weak_CV} (and thus of Proposition \ref{P:profile}) is complete.
\end{proof}
\begin{proof}[Proof of Lemma \ref{L:Pyt}]
 We first prove \eqref{A16}. We have
 $$ f_k=\sum_{j=1}^J e^{it_{j,k}\Delta}\varphi_j(h_{j,k}\cdot)+r_{J,k},$$
 and thus
 $$\|f_k\|^2_{H^1}=\sum_{j=1}^J \|\varphi_j\|^2_{H^1}+\|r_{J,k}\|^2_{H^1}+A_k+B_k,$$
 where 
 $$A_k=2\sum_{1\leq j<j'\leq J} \left(e^{it_{j,k}\Delta} \varphi_j(h_{j,k}\cdot),e^{it_{j',k}\Delta} \varphi_{j'}(h_{j',k}\cdot)\right),\quad  B_k=2\sum_{j=1}^J \left(e^{it_{j,k}\Delta} \varphi_j(h_{j,k}\cdot),r_{J,k}\right).$$
 The sum $A_k$ goes to $0$ as $k$ goes to infinity by the orthogonality of the profiles ensured by Lemma \ref{C:ortho} and \eqref{A12}. Moreover, the term $B_k$ equals
 $$2\sum_{j=1}^J  \left(\varphi_j,e^{-it_{j,k}\Delta}f_{k}(h_{j,k}^{-1}\cdot )-\varphi_j\right)_{H^1}-2A_k,$$
 which goes to $0$ as $k$ goes to infinity by \eqref{A13}.
 
 We next prove \eqref{A17}. By the refined Sobolev embedding \eqref{A1} applied to $s=\frac{n(p-1)}{2(p+1)}\in (0,1)$ and \eqref{Bs-Sigma}, 
 $$\|r_{J,k}\|_{L^{p+1}}\leq \|e^{it\Delta}r_{J,k}\|_{L^\infty L^{p+1}}\leq \|r_{J,k}\|_{H^1}^\frac2{p+1}\|e^{it\Delta}r_{J,k}\|_{L^\infty B^s}^\frac{p-1}{p+1}\leq \|r_{J,k}\|_{H^1}^\frac2{p+1}\|r_{J,k}\|_{\Sigma}^\frac{p-1}{p+1},$$
 so using \eqref{A15},
 $$\lim_{J\to\infty}\varlimsup_{k\to\infty} \|r_{J,k}\|_{L^{p+1}}=0.$$
 By the Poincar\'e-Sobolev inequality \eqref{PoiSob} and \eqref{A11}, 
 $$\sum_{j\geq 1}\|\varphi_j\|_{L^{p+1}}^{p+1}<\infty.$$
 We are thus reduced to prove that for a fixed $J$,
 $$\lim_{k\to\infty}\int \Big|\sum_{j=1}^J(e^{-it_{j,k}\Delta}\varphi_j)(h_{j,k}^{-1}x)\Big|^{p+1}\,dx- \sum_{j=1}^J\int \left| (e^{-it_{j,k}\Delta}\varphi_j)(h_{j,k}^{-1}x)\right|^{p+1}\,dx=0.$$
This last property follows from the inequality
\begin{equation}
 \label{ineg_R}
\forall (a_1,\ldots,a_J)\in [0,+\infty)^J,\quad \left| \sum_{j=1}^J a_j^{p+1}-\Big(\sum_{j=1}^J a_j\Big)^{p+1}\right|\leq C_J\sum_{\substack{1\leq j,j'\leq J\\j\neq j'}} a_j^pa_{j'},
 \end{equation} 
and the limit \eqref{A20} of Lemma \ref{C:ortho}, which concludes the proof.
 \end{proof}
 
\subsubsection{Nonlinear profiles and scattering}
\label{SS:NL_profile}
Let $\left(\varphi_j,(t_{j,k},h_{j,k})_k\right)_{j\geq 1}$ be a profile decomposition for a bounded sequence in $H^1$. Extracting subsequences, we can assume:
\begin{equation}
\label{lim_t_exist}
\forall j\geq 1,\quad \lim_{k\to\infty} t_{j,k}=\tau_{j}\in [-\infty,+\infty]. 
\end{equation} 
For any $j$, we denote by $U_j$ the \emph{nonlinear profile} associated to $\varphi_j$ and the sequence $(t_{j,k})_k$. This is by definition the unique solution of \eqref{NLS} such that $t_{j,k}\in I_{\max}(U_j)$ for large $k$ and
\begin{equation}
 \label{A23}
 \lim_{k\to\infty} \left\|e^{it_{j,k}\Delta}\varphi_j-U_j(t_{j,k})\right\|_{H^1}=0.
\end{equation} 
Assuming \eqref{lim_t_exist}, there always exists a nonlinear profile $U_j$: this follows from the local Cauchy theory if $\tau_j\in \RR$, and from the existence of wave operators (see Proposition \ref{P:wave_op}) if $\tau_{j}=\pm \infty$.
Note that if $T^{+}(U_j)$ is finite, then $\tau_j<T^+(U_j)$, and similarly, if $T^-(U_j)$ is finite, $T^-(U_j)<\tau_j$.
We next prove:
\begin{prop}
\label{P:scatt_profiles}
 Let $(f_k)_{k\geq 1}$ be a bounded sequence in $H^1$ that admits a profile decomposition $\left(\varphi_j,(t_{j,k},h_{j,k})_k\right)_{j\geq 1}$. Assume that for all $j\geq 1$, the corresponding nonlinear profile $U_j$ scatters forward in time. Then for large $k$, the solution $u_k$ of \eqref{NLS} with initial data $f_k$ at $t=0$ scatters forward in time. Furthermore,
 $$\varlimsup_{k\to\infty}\|u_k\|_{S^0(0,+\infty)}<\infty.$$
\end{prop}
\begin{proof}
This is a standard consequence of the long-time perturbation theory (Proposition \ref{P:LTPT}) applied to $u_k$ and $$u_{J,k}=\sum_{j=1}^J U_{j,k},\quad \mbox{where\,}\,U_{j,k}(t,x)=U_j(t+t_{j,k},h_{j,k}\cdot x).$$
We sketch the proof for $1<p\leq \frac{4}{n}+1$; recall that in this case $S^0(I)=L^{p+1}(I,L^{p+1}), N^0(I)=L^{\frac p{p+1}}(I,L^{\frac{p}{p+1}})$. We refer to \cite{FaXiCa11} for a very close proof, in the Euclidean setting, in the case $p>\frac{4}{n}+1$. 
 
 \emph{Step 1. Uniform bound on the $S^0$ norm.} We first prove that there is a constant $M>0$, depending on the sequence $(f_k)_k$, but not on $J$, such that
 \begin{equation}
  \label{A24}
  \forall J,\quad \varlimsup_{k\to\infty}\left\|u_{J,k}\right\|_{S^0(0,+\infty)}\leq M.
 \end{equation}

To this purpose we first use inequality \eqref{ineg_R},
 \begin{multline}
  \label{A25}
  \left|\int_0^{+\infty} \int |u_{J,k}|^{p+1}\,d\mu(x)\,dt-\sum_{j=1}^J \int_0^{+\infty} \int_{\HH^n} \left|U_{j,k}\right |^{p+1}\,d\mu(x)\,dt \right|\\
  \leq C\sum_{\substack{1\leq j,j'\leq J\\ j\neq j'}} \int_0^{+\infty}\int_{\HH^n} |U_{j,k}(x)|^p|U_{j',k}(x)|\,d\mu(x)\,dt. 
 \end{multline} 
 We next prove
 \begin{equation}
  \label{A26}
  j\neq j'\Longrightarrow
  \lim_{k\to\infty} \int_0^{+\infty} \int_{\HH^n} |U_{j,k}|^p|U_{j',k}|\,d\mu(x)\,dt=0.
 \end{equation} 
The term in the limit \eqref{A26} equals
\begin{equation*}
 \int_0^{+\infty} \int_{\HH^n} |U_j(t_{j,k}+t,h_{j,k}\cdot x)|^p|U_{j'}(t_{j',k}+t,h_{j,k}\cdot x)|\,d\mu(x)\,dt
\end{equation*} 
We first note that 
\begin{multline}
 \label{A28}
 \forall f,g \in L^{p+1}(\RR\times \HH^n), \\
 \lim_{k\to\infty} \int_{\RR} \int_{\HH^n} |f(t_{j,k}+t,h_{j,k}\cdot x)|^p|g(t_{j',k}+t,h_{j',k}\cdot x)|\,d\mu(x)\,dt=0.
\end{multline} 
Indeed, this is obvious, arguing on the supports, and using the pseudo-orthogonality  \eqref{A13} of the parameters, if $f$ and $g$ are compactly supported. The general case follows by density.

If $U_j$ and $U_{j'}$ are globally defined, \eqref{A26} follows immediately from \eqref{A28} with $f=U_j$ and $g=U_{j'}$. If $U_j$ and $U_{j'}$ are not globally defined backward in time, \eqref{A26} follows from \eqref{A28} with $f=\chi_{t\geq \tau_j}\,U_j$ and $g=\chi_{t\geq \tau_{j'}}\,U_{j'}$, where $\chi_{t\geq A}$ is the characteristic function of $[A,+\infty)$, and $\tau_j$, $\tau_{j'}$ are defined in \eqref{lim_t_exist}. The other cases are similar.

Combining \eqref{A25} and \eqref{A26}, we get
\begin{equation*}
\varlimsup_{k\to\infty}\left\|u_{J,k}\right\|_{S^0(0,+\infty)}^{p+1}= 
\varlimsup_{k\to\infty}\sum_{j=1}^J \left\|U_j(t_{j,k}+\cdot,h_{j,k}\cdot)\right\|_{S^0(0,+\infty)}^{p+1}=\sum_{j=1}^J \|U_j\|_{S(\tau_j,+\infty)}^{p+1},
\end{equation*}
which yields \eqref{A24} since $\sum_{j=1}^{+\infty} \|U_j\|_{S(\tau_j,+\infty)}^{p+1}$ is finite by \eqref{A11} and the small data theory for \eqref{NLS}.

 \emph{Step 2. End of the proof.} Fix $J\geq 1$ such that
 \begin{equation*}
  \varlimsup_{k\to\infty}\|e^{it\Delta}r_{J,k}\|_{S^0(\RR)}\leq \frac{\eps(M)}{4},
 \end{equation*} 
 where $\eps(M)$ is given by Proposition \ref{P:LTPT}. Recall the notation $u_{J,k}=\sum_{j=1}^J U_{j,k}$. Then
 $$i\partial_t u_{J,k}+\Delta u_{J,k}+|u_{J,k}|^{p-1}u_{J,k}=\underbrace{\bigg|\sum_{j=1}^J U_{j,k}\bigg|^{p-1}\sum_{j=1}^J U_{j,k}-\sum_{j=1}^J |U_{j,k}|^{p-1}U_{j,k}}_{e_{J,k}}.$$
 We have
 $$\left\|e_{J,k}\right\|_{N^0(0,+\infty)}^{1+\frac 1p}=\left\|e_{J,k}\right\|_{L^{1+\frac 1p}}^{1+\frac 1p}\leq C\int_0^{+\infty} \int_{\HH^n} \sum_{\substack{1\leq j,j'\leq J\\ j\neq j'}} \left||U_{j,k}|^{p-1}U_{j',k}\right|^{1+\frac 1p},$$
 which goes to $0$ as $k$ goes to infinity, using the pseudo-orthogonality \eqref{A13} of the sequence of parameters and a proof similar to the one in Step 1.
 
 Choosing $k_0$ large, so that for $k\geq k_0$
 $$\left\|e_{J,k}\right\|_{N^0(0,+\infty)}+\left\|e^{it\Delta} r_{J,k}\right\|_{S^0(0,+\infty)}\leq \frac{\eps(M)}{2},$$
 we obtain by using \eqref{Stric1} and \eqref{A23} that for $k\geq k_0$,
 $$\left\|e^{it\Delta}f_k-e^{it\Delta}u_{J,k}(0)\right\|_{S^0(0,+\infty)}=\left\|e^{it\Delta} r_{J,k}\right\|_{S^0(0,+\infty)}\leq \frac{\eps(M)}{2},$$
 and the results follows from long-time perturbation theory (Proposition \ref{P:LTPT}) applied to $u_k$ and $u_{J,k}$.
 \end{proof}

\subsection{Existence of the critical solution} 
\label{SS:critical}
In this section we shall prove Theorem \ref{P:critical}. We will also prove the nonradial version of this result:
\begin{prop}[Nonradial critical element]\label{P:critical2}Let $\lambda< \frac{(n-1)^2}4,1< p< 1+\frac 4{n-2}$. 
There exists a global  solution $u_c$ of equation \eqref{NLS} and a family $(h(t))_{t\in \RR}$ of elements of $\GG$ such that
$$\{u_c(t,h(t)\cdot), t\in\mathbb R\}$$
has compact closure in $H^1(\mathbb H^n)$,
$$E_\lambda(u_c(0))\leq E_\lambda(Q_{\lambda}),\quad \|u_c(0)\|_{\HHH_\lambda}\leq \|Q_{\lambda}\|_{\HHH_\lambda},$$
and, if 
$$E_\lambda(u(0))<  E_\lambda(u_c(0)),\quad \|u(0)\|_{\HHH_\lambda}\leq \|Q_{\lambda}\|_{\HHH_\lambda},$$
then the solution $u$ of equation \eqref{NLS} scatters in both time directions. 
\end{prop}

We will use the compactness/rigidity method initiated in \cite{KeMe06}.

 We fix $\lambda<\frac{(n-1)^2}{4}$. 
Let $0<\omega\leq 1$. We introduce the following set: 
$$\mathcal K_\omega=\Big\{f\in H^1(\HH^N)\;:\; E_{\lambda}(f)\leq \omega E_{\lambda}(Q_{\lambda})\text{ and }\|f\|_{\HHH_{\lambda}}^2\leq \|Q_{\lambda}\|^2_{\HHH_{\lambda}}\Big\}.$$
(note that $\KKK_{\omega}$ also depends on $\lambda$, which will be fixed in all this subsection).

Theorem \ref{T:global} and Lemma \ref{L:control} yield the following facts. 
First, the set $\mathcal K_\omega$ is invariant with respect to the nonlinear evolution \eqref{NLS}. Second, if $u_0\in\mathcal K_\omega$, then its evolution through equation \eqref{NLS} is global in time. 
Third, if $u_0\in \mathcal K_\omega$, then $\|u_0\|_{\HHH_{\lambda}}\leq \omega^\frac 12 \|Q_{\lambda}\|_{{\HHH_{\lambda}}}$. Therefore, for $\omega$ small enough, starting with $u_0\in\mathcal K_\omega$ we obtain a scattering solution of \eqref{NLS}, in view of the small data theory (Corollary \ref{C:small_data}). We can then define 
$$\omega_0=\sup\Big\{0<\omega\leq 1, u_0\in \mathcal \KKK_\omega \Longrightarrow \mbox{ the solution }u\mbox{ of }\eqref{NLS}\mbox{ scatters in both time directions}\Big\}.$$
Since $u_{\lambda}=e^{it\lambda}Q_{\lambda}$ is a non scattering solution of \eqref{NLS} it follows that $\omega_0\leq 1$. Note that if $\omega_0=1$, then this solution is a critical element in the sense of Proposition \ref{P:critical}, and Proposition \ref{P:critical} follows. 

We shall now focus on the remaining cases $\omega_0<1$ and prove:
\begin{prop}
\label{P:exist_crit}
Let $\lambda$, $\omega_0$ and $\KKK_{\omega_0}$ be as above. Assume $\omega_0<1$. Then there exists a solution $u_c$ of \eqref{NLS} such that
\begin{gather*}
u_{c}(0)\in \KKK_{\omega_0}\\
\|u_c\|_{S^0(-\infty,0)}=\|u_c\|_{S^0(0,+\infty)}=+\infty,
\end{gather*}
and there exists a $h:\RR\to \GG$ such that 
$$K=\Big\{u_c(t,h(t)\cdot),\; t\in \RR\Big\}$$
has compact closure in $H^1(\HH^n)$.
\end{prop}
Similarly, define $\widetilde{\KKK}_{\omega}$ as the subset of the elements of $\KKK_{\omega}$ that are radially symmetric, and define $\tilde{\omega}_{0}$ as $\omega_0$, replacing $\KKK_{\omega}$ by $\widetilde{\KKK}_{\omega}$. 
\begin{prop}
\label{P:exist_crit_rad}
Assume $\tilde{\omega}_{0}<1$.
Then there exists a radially symmetric solution $v_{c}$ of \eqref{NLS} such that
\begin{gather*}
v_{c}(0)\in \widetilde{\KKK}_{\tilde{\omega}_{0}}\\
\|v_{c}\|_{S^0(-\infty,0)}=\|v_c\|_{S^0(0,+\infty)}=+\infty,
\end{gather*}
and 
$$\widetilde{K}=\Big\{v_{c}(t),\; t\in \RR\Big\}$$
has compact closure in $H^1(\HH^n)$.
\end{prop}
Note that $\omega_0\leq \tilde{\omega}_{0}$.
We conjecture that if $p\geq 1+\frac{4}{n}$, $\omega_0=1$. We will show later (see Section \ref{S:rigidity}) that $\tilde{\omega}_{0}=1$ if $n=3$ and $p\geq \frac{7}{3}$, or if $n=2$ and $p\geq 3$. The proofs of Propositions \ref{P:exist_crit}  and \ref{P:exist_crit_rad} are by now standard. 
We give the proof of Proposition \ref{P:exist_crit} for the sake of completeness. In view of Remark \ref{R:radial}, the proof of Proposition \ref{P:exist_crit_rad} is the same, assuming that all the functions are radial and taking off the isometries $h(t)$. 

We first prove a preliminary result. 
We denote by $U(t)$ the nonlinear evolution \eqref{NLS}: if $u_0\in H^1$, $u(t)=U(t)u_0$ is the unique solution of \eqref{NLS}, with maximal time of existence $I_{\max}(u_0)=(T_-(u_0),T_+(u_0))$.

\begin{lemma}\label{L:cv}
Assume $\lambda<\frac{(n-1)^2}{4}$.
Let $(f_k)_k$ be a sequence in $H^1(\HH^n)$ such that 
$$\underset{k\rightarrow\infty}{\overline{\lim}} E_{\lambda}(f_k)=\omega_0E_{\lambda}(Q_{\lambda}),  \quad\|f_k\|_{\HHH_{\lambda}}\leq \|Q_{\lambda}\|_{\mathcal H_{\lambda}}$$ 
and 
$$\|U(t)f_k\|_{S^0(T_-(f_k),0)}\overset{k\rightarrow\infty}{\longrightarrow}\infty,\quad \|U(t)f_k\|_{S(0,T_+(f_k))}\overset{k\rightarrow\infty}{\longrightarrow}\infty.$$ 
Then there exists a subsequence of $(f_k)_k$ (that we still denote by $(f_k)_k$), a sequence $(h_k)_k\in \mathbb G^{\NN}$, and $V\in H^1(\HH^n)$ with $E_{\lambda}(V)=\omega_0E_{\lambda}(Q_{\lambda}), \|V\|_{\HHH_{\lambda}}\leq \|Q_{\lambda}\|_{\mathcal H_{\lambda}}$, such that
$$\|f_k(h_k\cdot)-V\|_{H^1}\overset{k\rightarrow\infty}\longrightarrow 0.$$
\end{lemma}

\begin{proof}
Extracting subsequences, we can assume by Proposition  \ref{P:profile} that the sequence $(f_k)_k$ has a profile decomposition 
$\left(\varphi_j;(t_{j,k},h_{j,k})_{k}\right)_{j\geq 1}$. By the Pythagorean expansion \eqref{A16} for large $k$, and since $\|f_k\|_{\HHH_{\lambda}}<\|Q_{\lambda}\|_{\mathcal H_{\lambda}}$, we obtain 
\begin{equation}
 \label{bound_phij}
 \forall j\geq 1,\quad \|\varphi_j\|^2_{\HHH_{\lambda}}\leq \|Q_{\lambda}\|_{\HHH_{\lambda}}^2.
\end{equation} 
By the Poincar\'e-Sobolev inequality and the value of its best constant \eqref{formule_Dp} we have
$$E_{\lambda}(\varphi_j)\geq \frac{\|\varphi_j\|_{\HHH_{\lambda}}}2^2\left(1-\frac{2}{p+1}\frac{\|\varphi_j\|_{\HHH_{\lambda}}^{p-1}}{\|Q_{\lambda}\|_{\HHH_{\lambda}}^{p-1}}\right),$$
so we get 
\begin{equation}
 \label{positive_energy}
\forall j,\quad \varphi_j\neq 0\Longrightarrow E_{\lambda}(\varphi_j)>0.
 \end{equation} 
 Combining the Pythagorean expansions \eqref{A16} and \eqref{A17} with the assumption $f_k\in \KKK_{\omega_0}$, we obtain that for all $J\geq 1$,
 \begin{equation}
  \label{exp_E_1}
 \frac{1}{2}\sum_{j=1}^J \|\varphi_j\|^2_{\HHH_{\lambda}}+\varlimsup_{k\to\infty} \left(\frac{1}{2}\|r_{J,k}\|^2_{\HHH_{\lambda}}-\sum_{j\geq 1}\frac{1}{p+1}\|e^{-it_{j,k}\Delta}\varphi_j\|_{L^{p+1}}^{p+1}\right)
 \leq \omega_0 E_{\lambda}(Q_{\lambda}).
 \end{equation} 
 Extracting again subsequences, we can assume that for all $j$, there exists a nonlinear profile $U_j$ associated to $(\varphi_j,(t_{j,k})_k)$ (see \S \ref{SS:NL_profile}). Using the conservation of the mass and energy for each of this nonlinear profiles, we can write \eqref{exp_E_1}:
\begin{equation}
  \label{exp_E_2}
 \sum_{j=1}^J E_{\lambda}(U_j) + \varlimsup_{k\to\infty} \left(\frac{1}{2}\|r_{J,k}\|^2_{\HHH_{\lambda}}-\sum_{j\geq J+1}\frac{1}{p+1}\|e^{-it_{j,k}\Delta}\varphi_j\|_{L^{p+1}}^{p+1}\right)\leq \omega_0 E_{\lambda}(Q_{\lambda}).
 \end{equation} 
Note that for large $J$,
$$\sum_{j\geq J+1}\|e^{-it_{j,k}\Delta}\varphi_j\|_{L^{p+1}}^{p+1}\leq C\sum_{j\geq J+1}\|\varphi_j\|_{H^1}^{p+1}$$
which is finite, independent of $k$ and goes to $0$ as $J\to\infty$. Furthermore, by \eqref{positive_energy}, $E_{\lambda}(U_j)$ is nonnegative (and positive if $\varphi_j\neq 0$).  
Thus \eqref{exp_E_2} implies:
\begin{equation*}
\sum_{j=1}^{+\infty} E_{\lambda}(U_j) \leq \omega_0 E_{\lambda}(Q_{\lambda}).
 \end{equation*} 
By \eqref{positive_energy}, if there are more than two indexes $j$ such that $\varphi_j\neq 0$, we obtain $E_{\lambda}(U_j)<\omega_0E_{\lambda}(Q_{\lambda})$ for all $j\geq 1$. In view of \eqref{bound_phij} and Theorem \ref{T:global}, we deduce that for all $j$, $U_j$ is globally defined and satifies $U_j(0)\in \KKK_{\omega}$ for some $\omega<\omega_0$. By the definition of $\omega_0$, all the nonlinear profiles $U_j$ scatter in both time directions. From Proposition \ref{P:scatt_profiles}, we deduce that $f_k$ scatters in both time directions and
\begin{equation}
\label{bound_fk_f0}
\varlimsup_{k\to\infty} \|U(t)f_k\|_{S^0(\RR)}<\infty,
\end{equation} 
which contradicts our assumptions.
Thus there is at most one nonzero profile, say $\varphi_1$, and $r_{J,k}=r_{1,k}$ for all $J\geq 1$. Going back to \eqref{exp_E_2}, we see that if 
$$\varlimsup_{k\to\infty}\|r_{1,k}\|_{H^1}>0,$$
then $E_{\lambda}(U_1)<\omega_0E_{\lambda}(Q_{\lambda})$. Arguing as before, we would obtain again that \eqref{bound_fk_f0} holds, a contradiction. Thus
$$\lim_{k\to\infty}\|r_{1,k}\|_{H^1}=0, \quad E_{\lambda}(U_1)=\omega_0E_{\lambda}(Q_{\lambda}).$$

Hence (letting $V=\varphi_1$, $t_{k}=-t_{1,k}$, $h_k=h_{1,k}^{-1}$),
$$\|f_k(h_k\cdot)-e^{-it_{k}\Delta}V\|_{H^1}\overset{k\rightarrow\infty}\longrightarrow 0.$$
It remains to prove that $t_k$ is bounded.

If $t_k\overset{k\rightarrow\infty}\longrightarrow-\infty$ then by using Strichartz and Sobolev estimates we get
$$\|e^{it\Delta} f_k\|_{S^0(0,\infty)}\leq \|e^{it\Delta} (f_k(h_k\cdot)-e^{-it_k\Delta}V)\|_{S^0(0,\infty)}+\|e^{it\Delta} V\|_{S^0(-t_k,\infty)}\overset{k\rightarrow\infty}\longrightarrow 0.$$
Corollary \ref{C:small_data} insures that for $k$ large, $U(t)f_k$ scatters forward in time in $H^1$, and its $S^0(0,\infty)$ norm is bounded from above by a constant independent of $k$. This contradicts the hypothesis, so the limit of $t_k$ cannot be $-\infty$. In the same manner the limit cannot be $\infty$. In conclusion the limit of the sequence $t_k$ is finite and we conclude by using the $H^1$ continuity of the free Schr\"odinger evolution. 
\end{proof}


\begin{proof}[Proof of Proposition \ref{P:exist_crit}]
\emph{Step 1.} Existence of $u_c$.

Since $\omega_0<1$, by its definition we obtain a sequence of numbers $\omega_k$ approaching $\omega_0$ with $\omega_0\leq\omega_k<1$ and a sequence of functions $f_k$ such that $E_{\lambda}(f_k)\leq \omega_kE_{\lambda}(Q_{\lambda})$, $\|f_k\|_{\HHH_{\lambda}}\leq \|Q_p\|_{\HHH_{\lambda}}$ whose global evolution $U(t)f_k$ through equation \eqref{NLS} satisfies
$$\|U(t)f_k\|_{S(\mathbb R)}=\infty.$$
There exists a sequence $t_k$ such that 
$$\|U(t-t_k)f_k\|_{S(-\infty,0)}\overset{k\rightarrow\infty}{\longrightarrow}\infty,\quad \|U(t-t_k)f_k\|_{S(0,\infty)}\overset{k\rightarrow\infty}{\longrightarrow}\infty.$$ 
To simplify notations, we denote by $f_k$ the translations in time $U(-t_k)f_k$. In view of the global in time result of Theorem \ref{T:global} we have $\|U(t)f_k\|_{\HHH_{\lambda}}<\|Q_{\lambda}\|_{\mathcal H_{\lambda}}$ for all $t\in\mathbb R$. We have 
$$\underset{k\rightarrow\infty}{\overline{\lim}} E_{\lambda}(f_k)\leq \omega_0E_{\lambda}(Q_{\lambda}),$$ 
and we claim that equality holds. Otherwise there exists $k$ such that $E_{\lambda}(f_k)<\omega_0E_{\lambda}(Q_{\lambda})$, and by definition of $\omega_0$ we obtain that $U(t)f_k$ scatters in both time directions, which is not true. 
Therefore we can apply Lemma \ref{L:cv} to conclude that there exists $u_{0c}\in H^1$ with $E_{\lambda}(u_{0c})=\omega_0E_{\lambda}(Q_{\lambda}),\|u_{0c}\|_{\HHH_{\lambda}}\leq \|Q_{\lambda}\|_{\mathcal H_{\lambda}}$ and a sequence $(h_k)\in \mathbb G^{\NN}$ such that
$$\|f_k(h_k\cdot)-u_{0c}\|_{H^1}\overset{k\rightarrow\infty}\longrightarrow 0.$$
By Proposition \ref{P:LTPT} applied to $U(t)f_k(h_k\cdot)$ and $U(t)u_{0c}$, and since
$$\lim_{k\to\infty}\|U(t)f_k\|_{S(0,\infty)}= \lim_{k\to\infty}\|U(t)f_k\|_{S(-\infty,0)}=\infty,$$
%
we obtain
$\|U(t)u_{0c}\|_{S(0,\infty)}=\|U(t)u_{0c}\|_{S(-\infty,0)}=\infty$. Thus $u_c=U(t)u_{0c}$ does not scatter in $H^1$ neither forward nor backward in time.

\emph{Step 2.} We show that there exists $h:\RR\to\mathbb G$ such that the set $\{u_c(t,h(t)\cdot), t\in\mathbb R\}$ has compact closure in $H^1$. 

By a standard lifting argument, it is sufficient to prove that for all sequence of times $(t_k)_k$, there exists a subsequence of $(t_k)_k$ (still denoted by $(t_k)_k$) and a sequence $(h_k)_k\in \mathbb G^{\NN}$ such that $(u(t_k,h_k\cdot))_k$ converges in $H^1$.

In view of Lemma \ref{L:control}, $u_{0c}$ satisfy the assumptions of the global existence result Theorem \ref{T:global}, so it follows that $\{u_c(t_k),k\in\mathbb N\}$ is a bounded set of $H^1$. Also, by the mass and energy conservations, $E_{\lambda}(u_c(t_k))=E_{\lambda}(u_{0c})=\omega_0E_{\lambda}(Q_{\lambda})$. From Step 1 we know that $U(t)u(t_k)$ does not scatter in $H^1$ neither forward nor backward in time. Then in view of Proposition \ref{P:scatt} we obtain that $\|U(t)u(t_k)\|_{S(0,\infty)}=\|U(t)u(t_k)\|_{S(-\infty,0)}=\infty$. Therefore we can apply Lemma \ref{L:cv} to obtain the existence of $V\in H^1$ and a sequence $(h_k)_k\in \mathbb G^{\NN}$ such that
$$\|u_c(t_k,h_k\cdot)-V\|_{H^1}\overset{k\rightarrow\infty}\longrightarrow 0.$$
This concludes the proof.
\end{proof}
\subsection{Mass-subcritical case}
\label{SS:mass_subscritical}
We conclude this section by proving Proposition \ref{P:stableGS}. We assume $n\geq 3$, $1<p<1+\frac{4}{n}$. By Proposition \ref{P:mass-sub}, there exists $\lambda<\frac{(n-1)^2}{4}$, $\alpha>0$ such that $Q_{\lambda}$ is a minimizer for \eqref{min2}. We will prove that $E_{\lambda}(v_c(0))<E_{\lambda}(Q_{\lambda})$ by contradiction, in the spirit of the proof of the stability of the orbital stability of the ground states by Cazenave and Lions \cite{CaLi82}. Assume 
\begin{equation*}
E_{\lambda}(v_c(0))=E_{\lambda}(Q_{\lambda}). 
\end{equation*} 
For $\beta>0$, we will consider $u_{\beta}$, the solution of \eqref{NLS} with initial data $u_{\beta}(0)=\beta Q_{\lambda}$. Then 
$$ E_{\lambda}(u_{\beta}(0))=\frac{\beta^2}{2}\|Q_{\lambda}\|_{\HHH_{\lambda}}^2-\frac{\beta^{p+1}}{p+1}\|Q_{\lambda}\|_{L^{p+1}}^{p+1},$$
and thus (using the equality $\|Q_{\lambda}\|_{\HHH_{\lambda}}^2=\|Q_{\lambda}\|_{L^{p+1}}^{p+1}$),
$$ \beta<1\Longrightarrow E_{\lambda}(u_{\beta}(0))<E_{\lambda}(Q_{\lambda})=E_{\lambda}(v_c)\text{ and }\|u_{\beta}(0)\|_{\HHH_{\lambda}}<\|Q_{\lambda}\|_{\HHH_{\lambda}}.$$
By the definition of $v_c$, we deduce that the solution $u_{\beta}$ scatters in both time directions if $\beta<1$. In particular, $u_{\beta}$ is global and 
\begin{equation}
\label{dispersion}
\lim_{t\to\infty}\|u_{\beta}(t)\|_{L^{p+1}}=0. 
\end{equation} 
Furthermore, by Theorem \ref{T:global}, again if $\beta<1$, 
\begin{equation}
 \label{bound_u_beta}
\forall t\in \RR,\quad 
\|u_{\beta}(t)\|_{\HHH_{\lambda}}\leq \|Q_{\lambda}\|_{\HHH_{\lambda}},\quad \|u_{\beta}(t)\|_{L^{p+1}}\leq \|Q_{\lambda}\|_{L^{p+1}}
\end{equation} 
Let $k$ be an integer, and $\beta_k=1-2^{-k}$. By \eqref{dispersion}, there exists $t_k$ such that 
\begin{equation}
 \label{dispersion_bis}
\|u_{\beta_k}(t_k)\|_{L^{p+1}}\leq 2^{-k}.
\end{equation} 
Let
$$f_k =\frac{1}{\beta_k}u_{\beta_k}(t_k).$$
Then, by mass conservation,
$$\|f_k\|_{L^2}= \frac{1}{\beta_k}\|u_{\beta_k}(0)\|_{L^2}=\|Q_{\lambda}\|_{L^2}=\alpha.$$
By energy conservation,
$$ E(\beta_k f_k)=E(u_{\beta_k}(0))\overset{k\to\infty}{\longrightarrow} E(Q_{\lambda})=e(\alpha),$$
Thus, using also \eqref{bound_u_beta}
$$E(f_k)=\left(\frac{1}{2}-\frac{\beta_k^2}2\right)\|\nabla f_k\|^2_{L^2}-\left(\frac{1}{p+1}-\frac{\beta_k^{p+1}}{p+1}\right)\|f_k\|^{p+1}_{L^{p+1}}+E(\beta_kf_{k})\overset{k\to\infty}{\longrightarrow} E(Q_{\lambda}).$$
Finally, we have obtained that $(f_k)_k$ is a minimizing sequence for the minimization problem \eqref{min2}. By Proposition \ref{P:mass-sub}, $f_k$ converges (extracting subsequences if necessary) to a minimizer, a contradiction with \eqref{dispersion_bis}. \qed

\section{The rigidity argument}
\label{S:rigidity}

In this subsection we shall prove the following proposition, which, together with Proposition \ref{P:exist_crit_rad} will imply the Theorem \ref{T:main}, \eqref{I:scattering}. 
\begin{prop}\label{prop-rigidity}Let $n\in\{2, 3\}$, $1+\frac 4n\leq p<1+\frac 4{n-2}$ and $\lambda<\frac{(n-1)^2}4$. 
Let $u$ be a radial solution of \eqref{NLS} such that $\overline{\{u(t)\}}$ is a compact subset of $H^1_{rad}$. If $u_0$ is radial, $ E_\lambda(u_0)< E_\lambda(Q_{\lambda})$ and $\|u_0\|_{\mathcal H_\lambda}\leq \|Q_{\lambda}\|_{\mathcal H_\lambda}$, then $u\equiv 0$.
\end{prop}

In order to prove this proposition we shall need some additional information. We first recall the classical virial formula:
$$\partial_t^2\int_{\mathbb H^n}|u(t)|^2\,r^2 =G(u(t)),$$
where
\begin{multline}
\label{def_G}
G(f)=16E(f)+8\int_{\mathbb H^n}|\nabla_{\mathbb S^{n-1}} u(t)|^2\frac{r\cosh r-\sinh r}{\sinh r^3} 
\\ -\int_{\mathbb H^n}|u(t)|^2\Delta_{\mathbb H^n}^2\,r^2-\int_{\mathbb H^n}|u(t)|^{p+1}\left(\frac{2(p-1)}{p+1}\Delta_{\mathbb H^n}\,r^2-\frac{16}{p+1}\right),
\end{multline} 
where
$$\Delta_{\mathbb H^n} \,r^2\,=2+2(n-1)\,\frac{r\cosh r}{\sinh r},$$
$$\Delta_{\mathbb H^n}^2 \,r^2\,=2(n-1)^2-2(n-1)(n-3)\frac{r\cosh r-\sinh r}{\sinh^3r}$$
for radial functions, we have:
\begin{multline}
 \label{def_G_bis}
G(f)=8\|f\|_{\mathcal H}^2+2(n-1)(n-3)\int_{\mathbb H^n}|f|^2\frac{r\cosh r-\sinh r}{\sinh^3r}\,d\mu(x)\\
-\frac{4(p-1)}{p+1}\int_{\mathbb H^n}|f|^{p+1}\left(1+(n-1)\,\frac{r\cosh r}{\sinh r}\right)\,d\mu(x).
\end{multline} 
Note that the third term is well defined for $u\in H^1$, in view of the following lemma, that will be of use also later.

\begin{lemma}\label{comp-emb} Let $n\geq 2$  and $1< p<1+\frac 4{n-2}$. The space $H^1_{rad}(\mathbb H^n)$ is compactly embedded in $L^{p+1}(\mathbb H^n)$ and in $L^{p+1}\left((1+(n-1)\frac{r\cosh r}{\sinh r})dr, \mathbb H^n\right)$.
\end{lemma}

\begin{proof}
By the change of function 
$$v(r)=\left(\frac{\sinh r}{r}\right)^{\frac{n-1}{2}}u(r),$$
wee see that it is enough to show that $H^1_{rad}(\mathbb R^n)$ is compactly embedded in $L^{p+1}(w(r)dr)$ with $w(r)=\left(\frac{r}{\sinh r}\right)^{\frac{(p-1)(n-1)}2}$ for the first embedding result and in $L^{p+1}(\tilde{w}(r)dr)$ with $\tilde w(r)=(1+(n-1)\frac{r\cosh r}{\sinh r}) w(r)$ for the second one. This follows immediately from the compact embedding of $H^1_{rad}(\RR^n)$  into $L^{p+1}(\RR^n)$ (see \cite[Compactness Lemma p. 570]{We82}). 
\end{proof}

We shall use the following crucial lemma.
\begin{lemma}\label{virielbound1var}
Let $n\in\{2, 3\}$, $1+\frac 4n\leq p<1+\frac 4{n-2}$ and $\lambda<\frac{(n-1)^2}4$. Then 
$$\inf_{\substack{f\in{H^1_{rad}},\\  E_\lambda(f)\leq E_\lambda(Q_{\lambda}),\\ \|f\|_{\mathcal H_{\lambda}}\leq \|Q_{\lambda}\|_{\mathcal H_\lambda}}}G(f)=0,$$
and the minimizing sequences converge (after extraction) in $\HHH$ to the constant zero function or to $e^{i\theta}Q$ for some $\theta\in\mathbb R$, $Q\in \QQQ_{\lambda}$.
 \end{lemma}

\begin{proof}
We denote $m$ the infimum and we consider a minimizing sequence $f_k$,
$$G(f_k)\overset{k\to\infty}{\longrightarrow} m.$$ 
Since $(f_k)_k$ is bounded in $H^1$, we can suppose that (up to a subsequence) there exists a radial weak limit $f$ of $f_k$ in $H^1$,
$$f_k\xrightharpoonup{k\to\infty} f\text{ in }H^1.$$
Thus $f_k\rightharpoonup{}f$ in $\HHH$ also. As a consequence,
\begin{equation}\label{fatou}
\liminf_{k\to\infty} \|f_k\|_{\mathcal H_\lambda}\geq \|f\|_{\mathcal H_\lambda} \text{ and }\liminf_{k\to\infty} \|f_k\|_{\mathcal H}\geq \|f\|_{\mathcal H}.
\end{equation}
By Lemma \ref{comp-emb} we obtain that 
$$f_k\overset{k\to\infty}{\longrightarrow} f\text{ in }L^{p+1},$$
and
$$\int_{\mathbb H^n}|f_k(x)|^{p+1}\left(1+(n-1)\,\frac{r\cosh r}{\sinh r}\right)\overset{k\to\infty}{\longrightarrow}\int_{\mathbb H^n}|f(x)|^{p+1}\left(1+(n-1)\,\frac{r\cosh r}{\sinh r}\right).$$
Finally, by H\"older's inequality,
$$\int_{\mathbb H^n}|f_k(x)-f(x)|^2\frac{r\cosh r-\sinh r}{\sinh^3r}\leq \|f_k-f\|_{L^{p+1}}^2\left\|\frac{r\cosh r-\sinh r}{\sinh^3r}\right\|_{L^\frac{p+1}{p-1}}\overset{k\to\infty}{\longrightarrow}  0,$$
and in particular
$$\int_{\mathbb H^n}|f_k(x)|^2\frac{r\cosh r-\sinh r}{\sinh^3r}\overset{k\to\infty}{\longrightarrow} \int_{\mathbb H^n}|f(x)|^2\frac{r\cosh r-\sinh r}{\sinh^3r}.$$
In view of the expression of $G$ it follows that $-\infty<m$ and  
\begin{equation}
\label{lim_fk_H}
8\|f_k\|^2_\mathcal H \longrightarrow m-(G(f)-8\|f\|^2_\mathcal H). 
\end{equation} 
Then, by \eqref{fatou}, we obtain $G(f)\leq m$ and $E_\lambda(f)\leq E_\lambda(Q_{\lambda}), \|f\|_{\mathcal H_\lambda}\leq \|Q_{\lambda}\|_{\mathcal H_\lambda}$. Thus $G(f)=m$, so $f$ is a minimizer. We shall distinguish five cases.

\medskip

\noindent{\em{Case 0: $f=0$.}} In this case, $m=0$. By \eqref{lim_fk_H}, 
$$\lim_{k\to\infty}\|f_k\|_{\HHH}=0.$$

\medskip

\noindent {\em{Case 1: $f\neq 0$, $E_\lambda(f)<E_\lambda(Q_{\lambda})$ and $\|f\|_{\mathcal H_\lambda}<\|Q_{\lambda}\|_{\mathcal H_\lambda}$.}} 

\smallskip

The set of functions such that $ E_\lambda(f)<E_\lambda(Q_{\lambda})$ and $\|f\|_{\mathcal H_\lambda}<\|Q_{\lambda}\|_{\mathcal H_\lambda}$ is open in $H^1$, so
$$0=\partial_\mu G(\mu f)\vert_{\mu=1}=2G(f)-\frac{4(p^2-1)}{p+1}\int_{\mathbb H^n}|f(x)|^{p+1}\left(1+(n-1)\,\frac{r\cosh r}{\sinh r}\right),$$
and since $f$ is not identically zero it follows that
$$m=G(f)>0=G(0),$$
so this case is excluded.

\medskip

\noindent {\em{Case 2: $f\neq 0$, $E_\lambda(f)=E_\lambda(Q_{\lambda})$ and $\|f\|_{\mathcal H_\lambda}=\|Q_{\lambda}\|_{\mathcal H_\lambda}$.}} 

\smallskip

By Theorem \ref{T:minimizers} it follows that $f=e^{i\theta}Q$, for some $\theta\in\mathbb R$, $Q\in \QQQ_{\lambda}$. 
The function $e^{it\lambda}Q$ is a solution of \eqref{NLS}, so the virial formula yields
\begin{equation*}
m=G(f)=G(Q)=0.
\end{equation*}
Since $\|f_k\|_{\HHH_{\lambda}}\leq \|Q_{\lambda}\|_{\HHH_{\lambda}}=\|Q\|_{\HHH_{\lambda}}$ for all $k$, and $f_k$ converges weakly to $f=e^{i\theta}Q$, we deduce
$$\lim_{k\to\infty}\|f_k\|_{\HHH_{\lambda}}=\|f\|_{\HHH_{\lambda}}$$
thus $(f_k)_k$ converges strongly to $e^{i\theta}Q$ in $H^1$ (and thus in $\HHH$). 

\medskip

\noindent{\em{Case 3: $f\neq 0$, $E_\lambda(f)=E_\lambda(Q_{\lambda})$ and $\|f\|_{\mathcal H_\lambda}<\|Q_{\lambda}\|_{\mathcal H_\lambda}$. }}

\smallskip

Since $\{\|f\|_{\mathcal H_\lambda}<\|Q_{\lambda}\|_{\mathcal H_\lambda}\}$ is an open set, it follows that $m$ is a local minimum of $G(f)$ under the constraint $f\in\mathcal H_{rad}$ with $E_\lambda(f)=E_\lambda(Q_{\lambda})$. Indeed, to avoid that $f^*$ is an isolated point in the set $E_\lambda(f)=E_\lambda(Q_{\lambda})$ we might argue as follows. There exists locally a curve through $f^*$ that is in the set $E_\lambda(f)=E_\lambda(Q_{\lambda})$: otherwise $f^*$ is a local extremum for $E_\lambda(f)$, so $-\Delta f^*-\lambda f^*-|f^*|^{p-1}f^*=0,$ which contradicts, by using Poincar\'e-Sobolev inequalities,
$$\|f\|_{\mathcal H_\lambda}^2-\|f\|_{L^{p+1}}^{p+1}\geq \|f\|_{\mathcal H_\lambda}^2-\|f\|_{\mathcal H_\lambda}^{p+1}\frac{\|Q_{\lambda}\|_{L^{p+1}}^{p+1}}{\|Q_{\lambda}\|_{\mathcal H_\lambda}^{p+1}}=\|f\|_{\mathcal H_\lambda}^2\left(1-\frac{\|f\|_{\mathcal H_\lambda}^{p-1}}{\|Q_{\lambda}\|_{\mathcal H_\lambda}^{p-1}}\right)>0.$$

Therefore we obtain the existence of a Lagrange multiplier $\mu$ such that $f$ solves
\begin{multline}\label{L1}
16\left(-\Delta f-\frac{(n-1)^2}{4} f\right)+4(n-1)(n-3)\frac{r\cosh r-\sinh r}{\sinh^3r}f\\
-4(p-1)|f|^{p-1}f\left(1+(n-1)\,\frac{r\cosh r}{\sinh r}\right)
=\mu\left(-\Delta f-\lambda f-|f|^{p-1}f\right).
\end{multline}

If $\mu\geq 0$ we multiply with $\overline{f}$, integrate and obtain
\begin{equation}\label{L2}
2G(f)-\frac{4(p-1)^2}{p+1}\int |f|^{p+1}\left(1+(n-1)\,\frac{r\cosh r}{\sinh r}\right)
\end{equation}
$$=\mu(\|f\|_{\mathcal H_\lambda}^2-\|f\|_{L^{p+1}}^{p+1})>0.$$
It follows that
$$G(f)>0=G(0),$$
which contradicts the fact that $f$ is a minimizer.

If $\mu<0$ we note that the equation on $f$ is of type
$$-\Delta f-gf-h|f|^{p-1}f=0,$$
with explicit variable radial coefficients $g(r)$ and $h(r)$: 
$$g(r)=\frac{4(n-1)^2-\lambda\mu}{16-\mu}{\color{blue}-}\frac{4(n-1)(n-3)}{16-\mu}\frac{r\cosh r-\sinh r}{\sinh^3r},$$
$$h(r)=\frac{1}{16-\mu}\left(4(p-1)\left(1+(n-1)\,\frac{r\cosh r}{\sinh r}\right)-\mu\right).$$
Multiplying the equation by $\overline{\varphi \,\partial_rf+\frac {\partial_r\varphi}2 \,f}$, integrating from $0$ to infinity and taking the real part, we obtain by integration by parts that
$$0=\int_0^\infty|\partial_r f|^2\left(\partial_r\varphi-(n-1)\frac{\cosh r}{\sinh r}\varphi\right)+|f|^2\left(-\frac{\partial_r^3\varphi}4+\frac{n-1}{4}\partial_r\left(\frac{\cosh r}{\sinh r}\partial_r\varphi\right)+\frac12\partial_r g\varphi\right)$$
$$+\int_0^\infty|f|^{p+1}\left(-\frac{p-1}{2(p+1)}h\partial_r\varphi+\frac 1{p+1}\partial_r h\,\varphi\right).$$
We choose $\varphi(r)=r\sinh r^{n-1}$, so that
\begin{multline*}
\partial_r^3\varphi=2(n-1)^2\sinh^{n-1}r+(n-1)(2n-5)\sinh ^{n-3}r\\ -(n-1)(n-3)r\cosh r\sinh^{n-4}r
+(n-1)^2\partial_r\left(\frac{\cosh^2 r}{\sinh^2 r}\varphi\right),
\end{multline*}
\begin{equation*}
\partial_r\left(\frac{\cosh r}{\sinh r}\partial_r\varphi\right)=(n-1)\sinh^{n-1}r+(n-2)\sinh^{n-3}r+(n-1)\partial_r\left(\frac{\cosh^2 r}{\sinh^2 r}\varphi\right),
\end{equation*}
and get, using that $\|f\|^2_{\HHH}=\int_{\HH^n} |\partial_rf|^2d\mu-\frac{(n-1)^2}{4}\int_{\HH^n}|f|^2$,
\begin{multline*}
0=\|f\|_{\mathcal H}^2+\int_{\mathbb H^n}|f|^2
\left(\frac{(n-1)(n-3)}{4}\frac{r\cosh r-\sinh r}{\sinh^3 r}+\frac{r}{2}\partial_r g\right)\\
+\int_{\mathbb H^n}|f|^{p+1}\left(-\frac{p-1}{2(p+1)}h(1+(n-1)\,\frac{r\cosh r}{\sinh r})+\frac 1{p+1}r\partial_r h\right).
\end{multline*}
It follows then that 
$$G(f)=\int |f|^{p+1}\left(\frac{4(p-1)}{p+1}\left(1+(n-1)\,\frac{r\cosh r}{\sinh r}\right)(h-1)-\frac{8}{p+1}r\partial_r h\right)-4\int |f|^2\,r\partial_rg.$$
In order to obtain that $G(f)>0$ we want to have the coefficients of $|f|^{p+1}$ and of $|f|^2$  positive. 
The coefficient of $|f|^{p+1}$ is  
$$\frac{16(p-1)}{(16-\mu)(p+1)}\left(\frac{(n-1)r^2}{\sinh^2r}\left((p-1)(n-1)\cosh ^2 r+2\right)+2(n-1)(p-4)\frac{r\cosh r}{\sinh r}+p-5\right).$$
From the behavior near $r=0$ we see that $p\geq1+\frac 4n$ is a necessary condition for positivity. Moreover, since the coefficient of $p$ is positive, in order to show that the function is positive for $p\geq 1+\frac 4n$, it is enough to show it for $p=1+\frac 4n$, which is equivalent to
$$(2n-2) r^2\cosh^2 r + nr^2 - (3n-4) r\cosh r\sinh r - 2 \sinh^2 r \geq 0.$$
This function vanishes at $r=0$ and its first four derivatives are
$$(4n-4) r\cosh^2 r + (4n-4) r^2\cosh r\sinh r + 2n r - 3n \cosh r\sinh r - (3n-4) r(\cosh^2 r+\sinh^2 r),$$
$$(4n-4)\cosh^2r + 4nr\cosh r\sinh r + (4n-4)r^2(\cosh^2 r+\sinh^2 r)+2n+(-6n+4)(\cosh^2 r+\sinh^2 r),$$
$$(-12n+8)\cosh r\sinh r + (12n-8)r(\cosh^2 r+\sinh^2 r) + (16n-16)r^2\cosh r\sinh r,$$
$$(80n-64)r\cosh r\sinh r + (16n-16)r^2(\cosh^2r+\sinh^2 r).$$
All these derivatives vanish at $r=0$ and the fourth derivative is positive. Therefore we have the initial inequality for all $r\geq 0$ and all $n$, so
$$m=G(f)>-4\int |f|^2r\partial_rg.$$
Since
$$\partial_rg={\color{blue}-}\frac{4(n-1)(n-3)}{16-\mu}\frac{r\sinh^2 r-3r\cosh^2 r+3\cosh r\sinh r}{\sinh^4r},$$
its sign is given by $n-3$, so in particular, in dimensions $n\leq 3$ we obtain 
$$m=G(f)>0=G(0),$$
which contradicts the fact that $f$ is a minimizer. Therefore this case is excluded. 

\medskip

\noindent{\em{Case 4: $f\neq 0$, $E_\lambda(f)<E_\lambda(Q_{\lambda})$ and $\|
f\|_{\mathcal H_\lambda}=\|Q_{\lambda}\|_{\mathcal H_\lambda}$. }} This case is excluded by \eqref{var}. 

\medskip

Summarizing we have obtained that $m=0$, that the only minimizers are the constant zero function and $e^{i\theta}Q$, for some $Q\in \QQQ_{\lambda}$ and $\theta\in\mathbb R$, and that minimizing sequences tend in $\HHH$ to a minimizer. 
\end{proof}
We are now able to prove Proposition \ref{prop-rigidity}.
\begin{proof}
We suppose that $u_0$ is not the constant null function.

Recall here that Theorem \ref{T:global} insures us that if the initial data satisfies to $ E_\lambda(u_0)<E_\lambda(Q_{\lambda})$ and $\|u_0\|_{\mathcal H_\lambda}\leq \|Q_{\lambda}\|_{\mathcal H_\lambda}$, then these properties will be preserved in time. In view of Lemma \ref{virielbound1var},
$$\inf_{t}G(u(t)) \geq 0.$$
and equality holds if there is a sequence of times $(t_n)$ such that $G(u(t_n))\rightarrow 0$ and $u(t_n)$ tends in $\mathcal H$ to the constant zero function or to $e^{i\theta}Q$ for some $Q\in \QQQ_{\lambda}$, $\theta\in\mathbb R$. It follows that there exists $\delta_0>0$ such that 
\begin{equation}
\label{G_positive}
G(u(t))>\delta_0,\quad \forall t\in\mathbb R.
\end{equation}
Indeed, otherwise there is a sequence of times $(t_n)$ such that $G(u(t_n))\rightarrow 0$ and $u(t_n)$ tends to the constant zero function or  
to $e^{i\theta}Q$ for some $Q\in \QQQ_{\lambda}$, $\theta\in\mathbb R$, strongly in $\mathcal H$, and thus, by compactness of $\overline{\{u(t)\}}$ in $H^1_{rad}$, strongly in $H^1$. In particular, $E_\lambda(u_0)=E_\lambda(u(t_n))$ which tends to $0$ or to $E_\lambda(Q_{\lambda})$. The second case contradicts the hypothesis $ E_\lambda(u_0)< E_\lambda(Q_{\lambda})$. In the first case, $E_\lambda(u_0)=0$ and by using the variational inequality \eqref{var} this contradicts the fact that we have supposed that $u_0$ is not the null function.

Now we recall that the classical virial computation yields for radial functions: 
\begin{multline}
\label{virial_h}
\partial_t^2\int_{\mathbb H^n}|u(t)|^2\,h\\
=\int_{\mathbb H^n} \left(|\partial_ru|^2-\frac{(n-1)^2}4|u|^2\right) \,4\partial_r^2 h-2\frac{p-1}{p+1}|u|^{p+1}\Delta h+|u|^2((n-1)^2\partial_r^2h-\Delta^2 h). 
\end{multline}
Let $\varphi$ be a smooth positive decreasing radial function supported in $B(0,2)$, valued $1$ in $B(0,1)$. We shall use the above formula with the weight $h_R(r)=r^2\varphi\left(\frac{r}{R}\right)$ and $R\geq 1$. Note that when all derivatives fall on $r^2$, then we recover $G(u(t))$ with the weight $\varphi\left(\frac{r}{R}\right)$. Otherwise at least one derivative in space falls on $\varphi\left(\frac{r}{R}\right)$, so the integral is restricted to the region $R\leq r\leq 2R$.  More precisely,
$$\partial_r^2 h_R=(\partial_r^2 r^2)\varphi\left(\frac{r}{R}\right)+\frac{4r}{R}\varphi'\left(\frac{r}{R}\right)+\frac{r^2}{R^2}\varphi''\left(\frac{r}{R}\right), $$
$$\Delta h_R=(\Delta r^2)\varphi\left(\frac{r}{R}\right)+(n-1)\frac{\cosh r}{\sinh r}\frac{r^2}{R}\varphi'\left(\frac{r}{R}\right)+\frac{4r}{R}\varphi'\left(\frac{r}{R}\right)+\frac{r^2}{R^2}\varphi''\left(\frac{r}{R}\right), $$
and similar computations show that
$$\left|((n-1)^2\partial_r^2h_R-\Delta^2 h_R)-\left((n-1)^2(\partial_r^2r^2) \varphi\left(\frac{r}{R}\right)-(\Delta^2 r^2)\varphi\left(\frac{r}{R}\right)\right)\right|\leq\frac{C}{r}1_{R\leq r\leq 2R}.$$ 
Therefore we obtain, using also the fact that $\varphi'\leq 0$,
\begin{multline*}
\partial_t^2\int_{\mathbb H^n}|u(t)|^2\,h_R=\partial_t\left(4\Im\int_{\mathbb H^n}u(t)\nabla\overline{u}(t)\nabla h_R\right)\\
\geq \Big(G(u(t))-C\int_{\HH^n\cap\{|x|\geq R\}}\left(|\nabla u(t,x)|^2+|u(t,x)|^2+|u|^{p+1}\right).
\end{multline*}
Therefore for $R$ large enough, we get, using the compactness of $\overline{\{u(t)\}}$ and \eqref{G_positive} that
$$\partial_t\left(4\Im\int_{\mathbb H^n}u(t)\nabla\overline{u}(t)\nabla h_R\right)\geq \frac{\delta_0}2.$$
Integrating in time, and using Cauchy-Schwarz inequality and then Hardy's inequality as above, we get
$$t\frac{\delta_0}2\leq C(1+R^2)\sup_{\tau\in(0,t)}\left(\left\|u(\tau)\right\|_{L^2(r\leq 2R)}^2+\|\nabla u(\tau)\|_{L^2(r\leq 2R)}^2\right)\leq C(R,\lambda) \sup_{\tau\in(0,t)}\|u(\tau)\|_{\mathcal H_\lambda}.$$
Therefore, we obtain a contradiction by letting $t$ go to infinity, and the Proposition follows. 

\end{proof}

\section{Blow-up}\label{section:blow-up}

\subsection{Previous blow-up results on hyperbolic space}
In this section we recall the known results on blow-up for equation \eqref{NLS}. These results are based on the method of Glassey \cite{Gl77,VlPeTa71}. 
If $u$ is a general (not necessarily radial) solution of \eqref{NLS}, and $h$ a radial weight, we have the following virial identity which generalizes \eqref{virial_h}:
$$\partial_t^2\int_{\mathbb H^n}|u(t)|^2\,h =\int_{\mathbb H^n} 4|\partial_ru|^2\partial_r^2 h +4\frac{|\nabla_{\mathbb S^{n-1}} u|^2}{\sinh^2 r}\partial_r h\frac{\cosh r}{\sinh r}-|u|^2\Delta^2 h-2\frac{p-1}{p+1}|u|^{p+1}\Delta h.$$

\begin{prop} \label{P:B-up}
Blow-up occurs in the following cases:
\begin{enumerate}
\item \label{I:bupi} \cite{Ba07} If $u_0$ is radial, of finite variance, $p\geq 1+\frac 4n$ and 
$$E(u_0)<\begin{cases}\frac{(n-1)^2}{8} \|u_0\|_{L^2}^2&\text{ if }n=2\text{ or }n=3\\
\frac{n(n-1)}{12}\|u_0\|_{L^2}^2&\text{ if }n\geq 4.
         \end{cases}$$
\item\label{I:bupii}\cite{MaZh07} If $u_0$ is of finite variance, not necessarily radial, $p\geq 1+\frac 4{n-1}$ and $E(u_0)<0$.
\end{enumerate}
\end{prop}
We refer to \cite{Ba07,MaZh07} for the proofs. Let us mention that both proofs are based on the  preceding virial identity, with $h(r)=r^2$ for \eqref{I:bupi}, and with $h(r)=\int_0^r\int_0^s\sinh^{n-1}\tau d\tau\frac{ds}{\sinh^{n-1}s}$, that satisfies $\Delta h=1$, for \eqref{I:bupii}.

In \cite{Ba07}, the blow-up sufficient condition is stated as:
$$E(u_0)<\inf_{r>0} \frac{\Delta^2_{\HH^n}r^2}{16}\|u_0\|_{L^2}^2.$$
Condition \eqref{I:bupi} follows, since
$$
\Delta_{\mathbb H^n}^2 \,r^2\,=2(n-1)^2-2(n-1)(n-3)\frac{r\cosh r-\sinh r}{\sinh^3r}$$
and
$$ \sup_{r>0} \frac{r\cosh r-\sinh r}{(\sinh r)^3}=\frac{1}{3}, \quad \inf_{r>0} \frac{r\cosh r-\sinh r}{(\sinh r)^3}=0.$$
\begin{rema}
By Proposition \ref{P:B-up} \eqref{I:bupi}, radial solutions with positive energy, small with respect to the $L^2$ norm always blow up, which seems better than the analoguous blow-up sufficient condition in the Euclidean setting. However, it is more natural to write this in term of the following conserved modified energy:
$$ E_m(u(t))=\frac{1}{2}\|u(t)\|^2_{\HHH}-\frac{1}{p+1}\|u(t)\|_{L^{p+1}}^{p+1}=E(u(t))-\frac{(n-1)^2}{8}\|u(t)\|^2_{L^2},$$
which takes into account the fact that the bottom of the spectrum of $-\Delta_{\HH^n}$ is $\frac{(n-1)^2}{4}$. 
The blow-up sufficient condition of Proposition \ref{P:B-up} \eqref{I:bupi} can be rewritten as
$$E_m(u_0)<\begin{cases}0&\text{ if }n=2\text{ or }n=3\\
-\frac{(n-3)(n-1)}{24}\|u_0\|_{L^2}^2&\text{ if }n\geq 4.
         \end{cases}$$
In dimension $n=2$ and $n=3$, we find a negative energy criterion similar to Glassey's criterion in the Euclidean setting. In higher dimension, we can only show that a stronger condition implies 	blow-up. Technically this is due to the term 
$$(n-1)(n-3)\int_{\HH^n} \frac{r\cosh r-\sinh r}{\sinh^3r}|u|^2$$
which has a bad sign in the virial identity, in dimension $n\geq 4$. Note that the assumption $n\in \{2,3\}$ in Theorem \ref{T:main} comes from the same technical problem (see e.g. the proof of Lemma \ref{virielbound1var}, Case 3).
\end{rema}

\subsection{Blow-up criterion in the finite variance case} 
In the particular case $h=r^2$, we can write the virial formula as
\begin{equation}
\label{virial_r2}
\partial_t^2\int_{\mathbb H^n}|u(t)|^2\,r^2 =G(u(t)),
\end{equation} 
with $G(f)$ defined in \eqref{def_G}.
In this section we obtain Theorem \ref{T:main}, \eqref{I:blow-up} in the finite variance case as a consequence of the following proposition. 
\begin{prop}\label{P:virielbound1bis}
Let $n\in\{2,3\}, p\geq  1+\frac 4n$ and $\lambda<\frac{(n-1)^2}4$. Let $u$ be a radial solution of \eqref{NLS} with $u_0$ in $H^1$. Then, if $ E_\lambda(u_0)\leq E_\lambda(Q_{\lambda})$ and $\|u_0\|_{\mathcal H_\lambda}> \|Q_{\lambda}\|_{\mathcal H_\lambda}$, 
$$\sup_{t}G(u(t)) \leq -16\left(E_{\lambda}(Q_{\lambda})-E_{\lambda}(u_0)\right).$$
 \end{prop}

Theorem \ref{T:main} \eqref{I:blow-up} in the finite variance case follows immediately from  Proposition \ref{P:virielbound1bis}. Indeed, in this case and under the assumptions of Theorem \ref{T:main} \eqref{I:blow-up}, the function $t\mapsto \int_{\HH^n} r^2|u(t)|^2$ is positive and strictly concave on $(T_-(u),T_+(u))$, which proves that $T_{+}(u)$  and $T_-(u)$ must be finite. We will treat the case of infinite variance in the next subsection. It remains to prove Proposition \ref{P:virielbound1bis}.

\begin{proof}[Proof of Proposition \ref{P:virielbound1bis}]
Theorem \ref{T:global} insures that if the initial data satisfies  $ E_\lambda(u_0)\leq E_\lambda(Q_{\lambda})$ and $\|u_0\|_{\mathcal H_\lambda}> \|Q_{\lambda}\|_{\mathcal H_\lambda}$, then these properties will be preserved in time. We thus define 
\begin{equation}
 \label{maximization1}
m=\sup_{\substack{f\in H^1_{rad},\\  E_\lambda(f)\leq E_\lambda(Q_{\lambda}),\\ \|f\|_{\mathcal H_\lambda}\geq \|Q_{\lambda}\|_{\mathcal H_\lambda}}}G(f).
\end{equation} 
We will show that $m=0$, then improve this estimate to get the conclusion of the proposition. We divide the proof into 4 steps. 

\smallskip

\noindent\emph{Step 1.}
We shall first prove that the supremum in \eqref{maximization1} can be restricted to the set $\big\{E_{\lambda}(f)=E_{\lambda}(Q_{\lambda})\big\}$. More precisely, we shall prove the following result.
\begin{lemma}\label{L:adjust} Assume $n\in\{2,3\}$, $p\geq 1+\frac{4}{n}$, $\lambda <\frac{(n-1)^2}{4}$. If $E_\lambda(f)< E_\lambda(Q_{\lambda}),\|f\|_{\mathcal H_\lambda}> \|Q_{\lambda}\|_{\mathcal H_\lambda}$ then there exists $f_*\in H^1$ such that $E_\lambda(f_*)=E_\lambda(Q_{\lambda}),\|f_*\|_{\mathcal H_\lambda}\geq \|Q_{\lambda}\|_{\mathcal H_\lambda}$ with $$G(f_*)>G(f)+16(E_{\lambda}(Q_{\lambda})-E_{\lambda}(f)).$$ 
\end{lemma}
\begin{proof}
We consider the family of functions $\{\sigma f\}_{\sigma\in [0,1]}$. If $E_\lambda(\sigma f)< E_\lambda(Q_{\lambda})$ for all $\sigma\in [0,1]$, then by Lemma \ref{L:control} and a simple continuity argument, it follows that $\|\sigma f\|_{\mathcal H_\lambda}> \|Q_{\lambda}\|_{\mathcal H_\lambda}$ for all $\sigma\in ]0,1]$, which is in contradiction with $\|Q_{\lambda}\|_{\mathcal H_\lambda}>0$. This yields the existence of $\sigma^*\in]0,1[$ such that $E_\lambda(\sigma^* f)=E_\lambda(Q_{\lambda})$ and $E_\lambda(\sigma f)< E_\lambda(Q_{\lambda})$ for $\sigma\in]\sigma^*,1]$. So for $\sigma\in]\sigma^*,1]$, again by Lemma \ref{L:control} and a simple continuity argument we have $\|\sigma f\|_{\HHH_\lambda}>\|Q_{\lambda}\|_{\mathcal H_\lambda}$. If $\sigma\in]\sigma^*,1]$, we have
$$\sigma\frac{\partial}{\partial\sigma}G(\sigma f)=16\|\sigma f\|_{\HHH_\lambda}^2-4(p-1)\int_{\mathbb H^n}|\sigma f|^{p+1}\left(1+(n-1)\,\frac{r\cosh r}{\sinh r}\right)\,dx$$
$$+16\int_{\mathbb H^n}|\sigma f|^2\left(\lambda-\frac{(n-1)^2}4+\frac{(n-1)(n-3)}4\frac{r\cosh r-\sinh r}{\sinh^3 r}\right)\,dx$$
$$< 16\|\sigma f\|_{\HHH_{\lambda}}^2-4n(p-1)\|\sigma f\|_{L^{p+1}}^{p+1}.$$
The inequality is strict since $f\neq 0$.
Hence, using that $p< 1+\frac{4}{n}$,
$$\frac{\partial}{\partial\sigma}G(\sigma f)\leq 16 \sigma\|f\|^2_{\HHH_{\lambda}}-16\sigma^p\|f\|^{p+1}_{L^{p+1}}.$$
Integrating between $\sigma^*$ and $1$, we get:
\begin{align*}
G(f)&< G(\sigma^*f)+16\left(\frac{1}{2}-\frac{{\sigma^*}^2}{2}\right)\|f\|^2_{\HHH_{\lambda}}-16\left(\frac{1}{p+1}-\frac{(\sigma^*)^{p+1}}{p+1}\right)\|f\|^{p+1}_{L^{p+1}}\\&=G(\sigma^*f)+16\left(E_{\lambda}(f)-E_{\lambda}(\sigma^*f)\right)=G(\sigma^*f)+16\left(E_{\lambda}(f)-E_{\lambda}(Q_{\lambda})\right).
\end{align*} 
Since $\|\sigma f\|_{\HHH_\lambda}>\|Q_{\lambda}\|_{\mathcal H_\lambda}$ for $\sigma\in]\sigma^*,1]$ we obtain that $\|\sigma^*f\|_{\HHH_\lambda}\geq \|Q_{\lambda}\|_{\mathcal H_\lambda}$ so we can set $f_*=\sigma^*f$.
\end{proof} 

\noindent \emph{Step 2. Maximizer for an equivalent maximization problem.}
Let
\begin{equation*}
H_\lambda(f)=G(f)-16 E_\lambda (f)+16 E_\lambda (Q_{\lambda}).
\end{equation*}
Note that $H_{\lambda}(f)=G(f)$ if $E_{\lambda}(f)=E_{\lambda}(Q_{\lambda})$.
In virtue of Lemma \ref{L:adjust} we obtain that $G(f)<H_{\lambda}(f)<G(f_*)$  if $E_{\lambda}(f)<E_{\lambda}(Q_{\lambda})$, $\|f\|_{\HHH_{\lambda}}>\|Q_{\lambda}\|_{\HHH_{\lambda}}$, and thus:
\begin{equation*}
m=\sup_{\substack{f\in H_{rad}^1\\  E_\lambda(f)= E_\lambda(Q_{\lambda})\\ \|f\|_{\mathcal H_\lambda}\geq \|Q_{\lambda}\|_{\mathcal H_\lambda}}}G(f)=\sup_{\substack{f\in H^1_{rad},\\ E_\lambda(f)\leq E_\lambda(Q_{\lambda}),\\ \|f\|_{\mathcal H_\lambda}\geq \|Q_{\lambda}\|_{\mathcal H_\lambda}}}H_\lambda(f),
\end{equation*} 
In this step we prove that there exists a maximizer $f$ for the maximization problem 
\begin{equation}
\label{maximization}
 m=\sup_{\substack{f\in H^1_{rad},\\ E_\lambda(f)\leq E_\lambda(Q_{\lambda}),\\ \|f\|_{\mathcal H_\lambda}\geq \|Q_{\lambda}\|_{\mathcal H_\lambda}}}H_\lambda(f)
\end{equation} 
and that it satisfies 
$$E_{\lambda}(f)=E_{\lambda}(Q_{\lambda}).$$
Note that
\begin{align*}
H_{\lambda}(f)&=\frac {16}{p+1}\|f\|_{L^{p+1}}^{p+1}-\frac{4(p-1)}{p+1}\int_{\mathbb H^n}|f|^{p+1}\left(1+(n-1)\,\frac{r\cosh r}{\sinh r}\right)\,dx+16 E_\lambda (Q_{\lambda})\\
&\qquad +8\int_{\mathbb H^n}|f|^2\left(\lambda-\frac{(n-1)^2}4+\frac{(n-1)(n-3)}4\frac{r\cosh r-\sinh r}{\sinh^3 r}\right)\,dx\\
&\leq 4\frac{4-n(p-1)}{p+1}\|f\|_{L^{p+1}}^{p+1}+16 E_\lambda (Q_{\lambda}),
\end{align*}
by our assumptions on $p$ and $n$. 
We consider now a maximizing sequence $f_k$,
$$H_\lambda(f_k)\overset{k\to\infty}{\longrightarrow} m.$$ 
From the above upper-bound on $H_\lambda(f_k)$ we obtain that the sequence $(f_k)_k$ is bounded in $L^{p+1}$. Since $E_\lambda(f_k)\leq E_\lambda(Q_{\lambda})$ it follows that $(f_k)_k$ is bounded in $H^1$, and we can suppose that (up to a subsequence) there exists a radial weak limit $f$ of $f_k$ in $H^1$,
$$f_k\xrightharpoonup{k\to\infty} f\text{ in }H^1.$$
By Fatou's lemma, since $\lambda<\frac{(n-1)^2}4$,
\begin{equation}\label{fatoubis}
\left(\lambda-\frac{(n-1)^2}4\right)\liminf_{k\to \infty} \|f_k\|_{L^2}^2\leq \left(\lambda-\frac{(n-1)^2}4\right)\|f\|_{L^2}^2.
\end{equation}
By the compactness Lemma \ref{comp-emb}, 
\begin{multline*}
\frac {16}{p+1}\|f_k\|_{L^{p+1}}^{p+1}-\frac{4(p-1)}{p+1}\int_{\mathbb H^n}|f_k(x)|^{p+1}\left(1+(n-1)\,\frac{r\cosh r}{\sinh r}\right)\,dx\\
\overset{k\to\infty}{\longrightarrow} \frac {16}{p+1}\|f\|_{L^{p+1}}^{p+1}-\frac{4(p-1)}{p+1}\int_{\mathbb H^n}|f|^{p+1}\left(1+(n-1)\,\frac{r\cosh r}{\sinh r}\right)\,dx, 
\end{multline*}
and, using also H\"older's inequality
$$\int_{\mathbb H^n}|f_k|^2\frac{r\cosh r-\sinh r}{\sinh^3 r}\,dx\overset{k\to\infty}{\longrightarrow} \int_{\mathbb H^n}|f|^2\frac{r\cosh r-\sinh r}{\sinh^3 r}\,dx.$$
In view of the expression of $H_\lambda$ it follows that $m<\infty$ and that $(\frac{(n-1)^2}4-\lambda)\|f_k\|_{L^2}$ converges. From \eqref{fatoubis} we obtain 
$$m=\lim_{k\rightarrow\infty}H_\lambda(f_k)\leq H_\lambda(f).$$ 
It remains to prove that $f$ satisfies $E_\lambda(f)=E_\lambda(Q_{\lambda})$, $\|f\|_{\mathcal H_\lambda}\geq \|Q_{\lambda}\|_{\mathcal H_\lambda}$.

By the weak convergence we obtain  
\begin{equation*}
\liminf_{k\to\infty} \|f_k\|_{\HHH_\lambda}\geq \|f\|_{\HHH_\lambda},
\end{equation*}
so combining this with the $L^{p+1}$ convergence,
$$E_\lambda(f)\leq \liminf_{k\to\infty} E_\lambda(f_k)\leq E_\lambda(Q_{\lambda}).$$
Moreover, since $f_k$ satisfy the constraints in \eqref{maximization}, it follows that $\|f_k\|_{L^{p+1}}\geq  \|Q_{\lambda}\|_{L^{p+1}}$ and by the $L^{p+1}$ convergence we get $\|f\|_{L^{p+1}}\geq  \|Q_{\lambda}\|_{L^{p+1}}$. Now using Poincar\'e-Sobolev inequality \eqref{PoiSob}, 
$$ \|Q_{\lambda}\|_{\mathcal H_\lambda}^2= \|Q_{\lambda}\|_{L^{p+1}}^{p+1}\leq \|f\|_{L^{p+1}}^{p+1}\leq D_{\lambda}^\frac{p+1}2  \|f\|_{\mathcal H_\lambda}^{p+1}=\|Q_{\lambda}\|_{\mathcal H_\lambda}^{1-p}\|f\|_{\mathcal H_\lambda}^{p+1},$$
so we get the second constraint 
$$\|f\|_{\HHH_\lambda}^2\geq \|Q_{\lambda}\|_{\mathcal H_\lambda}^2.$$
Therefore we have obtained that $f$ is a solution of the maximization problem \eqref{maximization}. 

To conclude this step, we must prove $E_\lambda(f)=E_\lambda(Q_{\lambda}).$ 
Indeed, if $f$ does not satisfy this constraint, then $E_\lambda(f)<E_\lambda(Q_{\lambda})$ and, letting $f_*$ be as in Lemma \ref{L:adjust}, we have
$$ H_{\lambda}(f_*)=G(f_*)>H_{\lambda}(f)=m,$$
a contradiction.

\smallskip

\noindent\emph{Step 3. Proof that the maximum is zero.} In the following we shall prove that $f=e^{i\theta}Q$ for some $\theta\in\mathbb R$, $Q\in \QQQ_{\lambda}$ which implies $m=0$. We suppose that $f\neq e^{i\theta}Q$ for any $\theta\in\mathbb R$ and any $Q\in \QQQ_{\lambda}$. By the definition of $\QQQ_{\lambda}$ and Theorem \ref{T:minimizers} we get $\|f\|_{\HHH_\lambda}^2> \|Q_{\lambda}\|_{\mathcal H_\lambda}^2.$ In particular, $f$ is a local maximizer for $G(f)$ under the only constraint $E_\lambda(f)=E_\lambda(Q_{\lambda}).$ We derive the equation
$$G'(f)=\mu E'_{\lambda}(f),$$
with the Lagrange multiplier $\mu\in\mathbb R$. This is precisely equation \eqref{L1}. In particular, \eqref{L2} writes
\begin{multline}
\label{new_L2}
(-16+\mu)(\|f\|_{L^{p+1}}^{p+1}-\|f\|_{\HHH_\lambda}^2)+16\int |f|^2\left(\lambda-\frac{(n-1)^2}4+\frac{(n-1)(n-3)}4\frac{r\cosh r-\sinh r}{\sinh^3 r}\right)\\
=\int |f|^{p+1}\left(4(p-1)\left(1+(n-1)\frac{r\cosh r}{\sinh r}\right)-16\right).
\end{multline}
Since $\|f\|_{L^{p+1}}>\|Q_{\lambda}\|_{L^{p+1}}$ we have
$$\|f\|_{L^{p+1}}^{p+1}-\|f\|_{\HHH_\lambda}^2=-2E_\lambda(f)+\frac{p-1}{p+1} \|f\|_{L^{p+1}}^2> -2E_\lambda(Q_{\lambda})+\frac{p-1}{p+1} \|Q_{\lambda}\|_{L^{p+1}}^2=0.$$
Since $p\geq 1+\frac 4n$, the right-hand side of \eqref{new_L2} is positive. 
So, in view of the hypothesis $n\in \{2,3\}$, $\lambda<\frac{(n-1)^2}{4}$, we must have $\mu> 16$. Then, recalling the computation in {\em{Case 3}} of the proof of Lemma \ref{virielbound1var}, but with the opposite sign for $16-\mu$,  we get
$$m=G(f)<4\int |f|^2r\frac{4(n-1)(n-3)}{16-\mu}\frac{r\sinh^2 r-3r\cosh^2 r+3\cosh r\sinh r}{\sinh^4r}.$$
For $n\leq 3$ we obtain 
$$m=G(f)<0=G(Q_{\lambda}),$$
which contradicts the definition of $m$.  

\smallskip 

\noindent\emph{Step 4. Conclusion of the proof.} By Lemma \ref{L:adjust}, for all $t$ in the domain of existence of $u$, there exists $u_*(t)$ with $E_{\lambda}(u_*)=E_{\lambda}(Q_{\lambda})$,$\|u_*\|_{\HHH_{\lambda}}>\|Q_{\lambda}\|_{\HHH_{\lambda}}$ and such that
$$ G(u(t))\leq G(u_*(t))+16\left(E_{\lambda}(u_0)-E_{\lambda}(Q_{\lambda})\right)\leq 16\left(E_{\lambda}(u_0)-E_{\lambda}(Q_{\lambda})\right),$$
since by Step 3, $G(u^*(t))\leq 0$. This concludes the proof of Proposition \ref{P:virielbound1bis}.
\end{proof}

\subsection{Blow-up criterion in the infinite variance case}
We next assume (in addition to the preceding assumptions on $p$, $n$ and $\lambda$), $1+\frac{4}{n}<p\leq 5$, and prove Theorem \ref{T:main} \eqref{I:blow-up} without the finite variance assumption. The proof relies on a localized version of the virial identity 
\eqref{virial_r2} in the spirit of \cite{OgTs91b}. To use this localized version, we need the following refinement of Proposition \ref{P:virielbound1bis}:
\begin{prop}\label{P:virielbound2bis}
Let $n\in\{2,3\}, p>  1+\frac 4n$ and $\lambda<\frac{(n-1)^2}4$. Let $u$ be a radial solution of \eqref{NLS} with $u_0$ in $\mathcal H_\lambda$ . Then, if $ E_\lambda(u_0)<E_\lambda(Q_{\lambda})$ and $\|u_0\|_{\mathcal H_\lambda}> \|Q_{\lambda}\|_{\mathcal H_\lambda}$, there exists $\delta>0$, depending only on the conserved mass and energy of $u$, such that for all $t$ in the maximal interval of existence $(T_-,T_+)$ of $u$,
\begin{equation}
\label{maj_G}
G(u(t)) \leq -\delta \|u(t)\|_{H^1}^2. 
\end{equation} 
 \end{prop}
\begin{proof}
 By Proposition \ref{P:virielbound1bis}, it is sufficient to prove \eqref{maj_G} when $\|\nabla u(t)\|_{L^2}$ is large, i.e.
\begin{equation}
\label{maj_G_bis} 
\exists M,\delta>0,\quad \forall t \in (T_-,T_+),\quad \|\nabla u(t)\|_{L^2}\geq M\Longrightarrow G(u(t))\leq -\delta \|\nabla u(t)\|_{L^2}^2.
\end{equation} 
Using the definition \eqref{def_G_bis} of $G(u(t))$ and the assumption $n\leq 3$, we obtain 
\begin{multline*}
G(u(t))\leq 8\|u(t)\|^2_{\HHH}-\frac{4n(p-1)}{p+1}\|u(t)\|^{p+1}_{L^{p+1}}\\
=4n(p-1)E(u(t))+(8-2n(p-1))\|\nabla u(t)\|^2_{L^2}-\frac{(n-1)^2}{4}\|u(t)\|^2_{L^2},
\end{multline*}
and \eqref{maj_G_bis} follows, since by our assumptions $8-2n(p-1)<0$.
\end{proof}
We next prove Theorem \ref{T:main} \eqref{I:blow-up}. We will only sketch the proof, which is close to the corresponding proof on $\RR^n$ once  \eqref{maj_G} is known.  

Let $\varphi:(0,\infty)\to (0,+\infty)$ be a smooth function such that $\varphi(r)=r^2$ if $r\leq 1$, $\varphi(r)$ is constant for $r\geq 2$ and $\varphi''(r)\leq 2$ for all $r>0$. Let $R\geq 1$ and $h_R(r)=R^2\varphi(r/R)$. Combining the virial identity \eqref{virial_h} with $h(r)=h_R(r)$, and the definition \eqref{def_G_bis} of $G$, we obtain
\begin{equation}
\label{virial_abc}
\frac{\partial^2}{\partial t^2}\int |u(t)|^2 h_R-G(u(t))=a+b+c,
\end{equation}
with 
\begin{align*}
 a&=4\int_{\HH^n} \left(|\partial_r u|^2-\frac{(n-1)^2}{4}|u|^2\right) \left(\varphi''(r/R)-2\right),\\b&= -2\frac{p-1}{p+1}\int_{\HH^n} |u|^{p+1}\left(\left(\varphi''(r/R)-2\right)+(n-1)\frac{\cosh r}{\sinh r}\left(R\varphi'(r/R)-2r\right)\right)\\
c&=\int_{\HH^n} |u|^2\bigg[-\frac{1}{R^2}\varphi^{(4)}\left(r/R\right)-\frac{2(n-1)}{R} \frac{\cosh r}{\sinh r}\varphi^{(3)}\left(r/R\right)\\
&\qquad-\frac{(n-1)(n-3)}{\sinh^2(r)}\left(\varphi''\left(r/R\right)-2\right)+\frac{\cosh r}{\sinh^3 r}(n-1)(n-3)\left(R\varphi'\left(r/R\right)-2r\right)\bigg].
\end{align*}
By the choice of $\varphi$, the integrand in the definitions of $a$, $b$ and $c$ is zero for $r\leq R$. We claim
\begin{gather}
 \label{bound_abc}
a\leq \frac{C}{R}\|u\|_{L^2}^2,\quad |b|\leq Ce^{-R/C} \|u\|^{\frac{p-1}{2}}_{H^1}\|u\|^{\frac{p+3}{2}}_{L^2},\quad |c|\leq \frac{C}{R}\|u\|_{L^2}^2.
\end{gather}
We first assume \eqref{bound_abc} and prove that $u$ blows up in finite time. Combining Proposition \ref{P:virielbound2bis}, \eqref{virial_abc} and \eqref{bound_abc}, we obtain that for all $t$ in the domain of existence of $u$,
$$ \frac{\partial^2}{\partial t^2} \int_{\HH^n} |u(t)|^2h_R\leq -\delta \|u(t)\|_{H^1}^2+Ce^{-R/C} \|u(t)\|_{H^1}^{\frac{p-1}{2}}\|u(t)\|_{L^2}^{\frac{p+3}{2}}+\frac{C}{R}\|u(t)\|_{L^2}^2.$$
Using the conservation of the mass and Young's inequality together with the assumption $p\leq 5$ we deduce (for a constant $C>0$ that depends on the mass and energy of $u$),
$$ \frac{\partial^2}{\partial t^2} \int_{\HH^n} |u(t)|^2h_R\leq -\delta \|u(t)\|_{H^1}^2+Ce^{-R/C} \|u(t)\|_{H^1}^{2}+\frac{C}{R}.$$
Note that $\|u(t)\|_{H^1}^2$ is bounded from below (by the conserved mass of $u$). Chosing $R$ large, we obtain 
$$ \frac{\partial^2}{\partial t^2} \int_{\HH^n} |u(t)|^2h_R\leq -\frac{\delta}{2} \|u(t)\|_{H^1}^2\leq -\frac{\delta}{2}M(u).$$
 Thus $\int |u(t)|^2h_R$ is a positive, strictly concave function on the domain of existence of $u$, which proves that $u$ blows up in finite time.

It remains to prove \eqref{bound_abc}. 

The bound on $c$ is straightforward using that the integrand in the definition of $c$ is zero for $r\leq R$.

To bound $a$, we let $w=(\sinh r)^{\frac{n-1}{2}} u$. Then
$$\left(|\partial_r u|^2 -\frac{(n-1)^2}{4} |u|^2\right)(\sinh r)^{n-1}=|\partial_r w|^2-\frac{n-1}{2}\frac{\cosh r}{\sinh r}\frac{\partial}{\partial r}|w|^2.$$
Hence, using that $4\varphi''(r/R)-8$ is nonpositive for all $r>0$, and $0$ for $r\leq R$,
$$ a\leq -\frac{n-1}{2}\int_0^{\infty} \frac{\cosh r}{\sinh r}\frac{\partial}{\partial r}|v|^2 \left(4\varphi''(r/R)-8\right)\,dr\leq \frac{C}{R}\int_0^{\infty} |v|^2\,dr\leq\frac{C}{R}\int_{\HH^n}|u|^2.$$

Finally, using that for $r\geq 1$,
\begin{multline*}
|u(t,r)|^2=\left|\int_{r}^{\infty}\frac{\partial}{\partial \rho} |u(t,\rho)|^2\,d\rho\right|\\
\leq \frac{1}{(\sinh r)^{n-1}} \sqrt{\int_r^{\infty} |\partial_{\rho}u(t,\rho)|^2\,(\sinh \rho)^{n-1}\,d\rho}\sqrt{\int_r^{\infty} |u(t,\rho)|^2\,(\sinh \rho)^{n-1}\,d\rho}\\
\leq Ce^{-(n-1)r}\|u(t)\|_{H^1}\|u(t)\|_{L^2},\quad\qquad
\end{multline*}
we obtain the bound of $b$ in \eqref{bound_abc} by explicit computation. The proof is complete

\bibliographystyle{abstract} 
\bibliography{toto}

\end{document}